\documentclass{lms}
\usepackage{tikz-cd}
\usepackage{ragged2e}
\usepackage{mathrsfs}
\usepackage{amsmath}
\usepackage{amssymb}
\usepackage{cite}
\newtheorem{theorem}{Theorem}[section] 
\newtheorem{lemma}[theorem]{Lemma}     
\newtheorem{corollary}[theorem]{Corollary}
\newtheorem{proposition}[theorem]{Proposition}

\newnumbered{assertion}{Assertion}    
\newnumbered{conjecture}{Conjecture}  
\newnumbered{definition}{Definition}
\newnumbered{hypothesis}{Hypothesis}
\newnumbered{remark}{Remark}
\newnumbered{note}{Note}
\newnumbered{observation}{Observation}
\newnumbered{problem}{Problem}
\newnumbered{question}{Question}
\newnumbered{algorithm}{Algorithm}
\newnumbered{example}{Example}
\newunnumbered{notation}{Notation} 

\title[]
 {Deformation families of open Calabi-Yau manifolds and Steinness} 
\dedication{}
\author{Fan Xu}
\classno{32G05 (Primary) 32E10, 14J27 (Secondary)}
\extraline{}

\begin{document}
\maketitle
\begin{abstract}
This paper is to construct deformation families of open Calabi-Yau manifolds together with some discussions on Steinness. Firstly, we construct a nontrivial compactifiable deformation of open Calabi-Yau manifolds from rational elliptic surface with global sections. 
Secondly, we give a condition for the existence of an elliptic fibration with global sections on the complex analytic family of rational elliptic surface with global sections, construct a holomorphic involution on the compactifiable deformation family and claim the non-Steinness for all the compactifiable deformation families constructed. Finally, we give some discussions on the non-Steinness of the quasi-projective variety from $\text{CP}^2 \text{$\#$9} \overline{\text{CP}}^2$ corresponding to a conjecture by Marco Brunella and construct deformation families whose fibers were conjectured to be Stein by Andreas Höring and Thomas Peternel.
\end{abstract}

\section{Introduction} 

\noindent
This is a continuation of the previous paper \cite{xu:2025}. This paper is motivated by the gravitational instantons from rational elliptic surfaces with global sections constructed by Hans-Joachim Hein in \cite{Hein:2012}.

Firstly, we construct a compactifiable deformation family of open Calabi-Yau manifolds in the sense of \cite{Gasparim:2021} by Edoardo Ballico, Elizabeth Gasparim and Francisco Rubilar.

Let $Y:=\{\tau | \tau\in \mathbb{C} \ \& \ Im\tau>0 \}$. For a lattice $\left \langle 1,\tau\right \rangle$ with $\tau\in Y$, $C_0(\tau)=\mathbb{C}/\left \langle 1,\tau\right \rangle$ is an elliptic curve. In addition, fixing $\tau_0 \in Y$, let $C_0=\mathbb{C}/\left \langle 1,\tau_0 \right \rangle$. Moreover, let $f_\tau$ be a holomorphic embedding map from $C_0(\tau)$ to $\text{CP}^2$. Here, $f_\tau$ can be derived from the embedding map constructed in [1, Introduction] through an algebraic isomorphism. Let $f_\tau(C_0(\tau))=\{[x:y:z] \in \text{CP}^2 | f_{\tau,\infty}(x,y,z)=0 \}$. Here, the coefficients of $f_{\tau,\infty}$ are induced by the  Weierstrass elliptic function in \cite[Introduction]{xu:2025} and the algebraic isomorphism.

Now $f_\tau(C_0(\tau))$ has a fixed inflexion point $Q_\tau$ and $f_\tau^{-1}(Q_\tau)$ is the zero point of $C_0(\tau)$ for $\tau\in Y$. One typical example of the embedding map is constructed through the Weierstrass elliptic function which can be found in \cite[Introduction]{xu:2025}.

Through blowing up $\text{CP}^2$ at nine points in the set $\{p_1,p_2,...,p_9\} \subset f_\tau(C_0(\tau))$, we can get a complex surface $S\cong\text{CP}^2 \text{$\#$9} \overline{\text{CP}}^2$. Let $C \subset S$ be the strict transform of $f_\tau(C_0(\tau))$. 

If $\sum_{i=1}^9 f_\tau^{-1} (p_i)=0$ mod $\left \langle 1,\tau\right \rangle$, then $S$ is a rational elliptic surface with global sections \cite[Theorem~1]{Fujimoto:1990}.

Fixing $\tau_0$ and the condition $\sum_{i=1}^9 f_\tau^{-1} (p_i)=0$ mod  $\left \langle 1,\tau\right \rangle$, we can construct a complex analytic family $(\widetilde{\mathcal{F}},T,\widetilde{\pi})$ of rational elliptic surfaces with global sections over a 9-dimensional complex manifold $T=U\times W_1 \times \cdots \times W_8$ through changing the variable $\tau \in U$ and the positions of the eight blow-up points in the sets $W_1,\cdots,W_8$ respectively. 

Here, let $ o=(\tau_0,q_1,q_2,...,q_8) \in  T \subset \mathbb{C}^9$ be the center of $T$. U is a sufficiently small disc neighborhood of $\tau_0$. $W_\upsilon$ is a sufficiently small disc neighborhood of $q_\upsilon$ in the universal cover $\mathbb{C}$ of $C_0 (\tau)$ and $f_{\tau,\upsilon}$ is the natural projection mapping $W_\upsilon$ to $W_\upsilon / \langle 1,\tau  \rangle  $ for $\upsilon \in \{1,\cdots,8\}$ and $\tau \in U$. In addition, let $f_{\tau,9}:\mathbb{C} \rightarrow \mathbb{C}/ \langle 1,\tau \rangle $ be the natural projection for $\tau \in U$.   

Let $S_0 \cong \text{CP}^2 \text{$\#$9} \overline{\text{CP}}^2$ be the blow-up of $\text{CP}^2$ at the nine points on $f_{\tau_0}(C_0)$  corresponding to $o \in \mathbb{C}^9$. Here, $j:S_0 \setminus C_0 \rightarrow \widetilde{\pi}^{-1} (o)$ is defined by $j(x)=(x,\tau_0,q_1,q_2,...,q_8)$ for $x \in S_0 \setminus C_0$. The quasi-projective varieties from the rational elliptic surfaces with global sections have been proved to be complete open Calabi-Yau manifolds in \cite{Hein:2012} by Hans-Joachim Hein.

Then for an appropriate complex manifold $T$, we can get a compactifiable deformation family as follows.

\begin{theorem}
\label{thm1.1}
$(\widetilde{\mathcal{F}},\widetilde{\mathcal{S}},T,\widetilde{\pi},o,j)$
is a nontrivial deformation family of open Calabi-Yau manifolds.
\end{theorem}

Secondly, we give a condition for the existence of an elliptic fibration with global sections on the complex analytic family $(\widetilde{\mathcal{F}},T,\widetilde{\pi})$, construct deformation families admitting a holomorphic involution and prove that $\widetilde{\mathcal{F}}\setminus \widetilde{\mathcal{S}}$ is not Stein.

Under some special conditions, we can derive a cubic curve in $\text{CP}^2$ intersecting $f_\tau (C_0 (\tau))$ at the points in the set $\{p_1,\cdots,p_9\}$. The coefficients of the cubic curve are $Q_{\tau,1},\cdots,Q_{\tau,9}$. For $p \in (M^2\cap M^3)\cup(M^4\setminus M^3) \subset \mathbb{C}^9$, we get a function $P_\tau:(M^2\cap M^3)\cup(M^4\setminus M^3) \rightarrow \text{CP}^8$ defined by $P_\tau (p)=[Q_{\tau,1}: \cdots :Q_{\tau,9}] \in \text{CP}^8$. Let $W=\{(p_1,\cdots,p_9 )|(p_1,\cdots,p_8 )=W_1 \times \cdots \times W_8 \quad \& \quad \sum_{i=1}^9 f_{\tau,i} (p_i)=0\}$.

Then there is a corollary as follows.

\begin{corollary}
\label{cor1.2}
If T is chosen properly such that $P_\tau$ is a holomorphic function on $W \subset (M^2\cap M^3)\cup(M^4\setminus M^3)$, then there exists an elliptic fibration with global sections on the complex analytic family $(\widetilde{\mathcal{F}},T,\widetilde{\pi})$.
\end{corollary}

In addition, we have a corollary as follows.

\begin{corollary}
\label{cor1.3}
For the deformation family constructed in Corollary~\ref{cor1.2}, if all the singular fibers are of Kodaira type $I_1$ and $C_{\tau,p}$ is a smooth elliptic curve for $p \in W$ and $\tau \in U$, then there exists a holomorphic involution on the compactifiable deformation family $(\widetilde{\mathcal{F}},\widetilde{\mathcal{S}},T,\widetilde{\pi},o,j)$.
\end{corollary}

Here, the involution map is induced by an involution on the lattice $C_0 (\tau)$ which is connected to mirror symmetry through the Borcea-Voisin Construction proposed in \cite{Borcea:1996}. Moreover, this involution map also connects with the hypersurface involution on the rational elliptic surface with global sections. 

Moreover, we present an example satisfying the condition of Corollary~\ref{cor1.2} derived from \cite[Example~3.3]{Fujimoto:1990}. Then we can construct a holomorphic involution on the compactifiable deformation family if the conditions proposed in Corollary~\ref{cor1.3} are all satisfied. In particular, if the center point $ o =(\tau_0,q_1,q_2,\cdots,q_8)$ of T satisfies the condition that $q_i=-q_{10-i}$ for $i \in \{2,\dots, 5\}$ with $\sum_{i=1}^9 f_{\tau,i}(q_i)=0$, then the elliptic fibration on $\widetilde{\mathcal{F}}$ with global holomorphic sections can not be derived directly.

In addition, it is easy to prove that $S\setminus C$ is not Stein in our case. Then the fibers of $(\widetilde{\mathcal{F}},\widetilde{\mathcal{S}},T,\widetilde{\pi},o,j)$ are all not Stein. So the compactifiable deformation families of open Calabi-Yau manifolds constructed above are all not Stein.

Finally, we prove the non-Steinness of the quasi-projective varieties from $\text{CP}^2 \text{$\#$9} \overline{\text{CP}}^2$ in almost every case corresponding to a conjecture proposed in \cite{Marco:2010} by Marco Brunella, discuss the remaining work and give deformation families in the sense of \cite{Gasparim:2021} whose fibers were conjectured to be Stein in \cite{Höring:2024}. Here, the nine blow-up points on $\text{CP}^2$ corresponding to the complex surface $S \cong \text{CP}^2 \text{$\#$9} \overline{\text{CP}}^2$ can be infinitely near and $f_\tau$ is an arbitrary holomorphic embedding map.

Classifying the value of $\sum_{i=1}^9 f_\tau^{-1}(p_i)$ corresponding to the positions of the nine blow-up points, the quasi-projective variety $S\setminus C$ can be proved to be not Stein if $\sum_{i=1}^9 f_\tau^{-1}(p_i)=m_1+m_2 \tau$ mod $\left \langle 1,\tau\right \rangle$ with $m_1+m_2 \tau=0$, of order $m$ or $(m_1,m_2)$ being a Diophantine number pair \cite[Definition~2.2]{xu:2025}.

If $\sum_{i=1}^9 f_\tau^{-1}(p_i)=0$ mod $\left \langle 1,\tau\right \rangle$, taking use of the elliptic fibration, we can prove the non-Steinness of $S\setminus C$ easily.

In addition, if $\sum_{i=1}^9 f_\tau^{-1}(p_i)$ is of order $m \geq 2$, taking use of logarithmic transformation and the result above, we can prove the non-Steinness of $S\setminus C$.

Moreover, if $\sum_{i=1}^9 f_\tau^{-1}(p_i)=m_1+m_2 \tau$ mod $\left \langle 1,\tau\right \rangle$ with $(m_1,m_2)$ being a Diophantine number pair, taking use of \cite[Theorem~1.7]{Koike and Uehara:2010}, $S\setminus C$ is not Stein.

Then we have a proposition as follows.

\begin{proposition}
\label{pro1.4}
For any holomorphic embedding $f_\tau$, the quasi-projective variety $S\setminus C$ is not Stein when $\sum_{i=1}^9 f_\tau^{-1}(p_i)=0$ mod $\left \langle 1,\tau\right \rangle$, of order $m \geq 2$ or $m_1+m_2 \tau$ mod $\left \langle 1,\tau\right \rangle$ with $(m_1,m_2)$ being a Diophantine number pair.
\end{proposition}

In addition, we can construct a smooth Hermitian metric whose curvature form is non-negative on the anticanonical line bundle of S when $\sum_{i=1}^9 f_\tau^{-1}(p_i)=0$, of order $m \geq 2$ or $m_1+m_2 \tau$ mod $\left \langle 1,\tau\right \rangle$ with $(m_1,m_2)$ being a Diophantine number pair.

Now for the non-Steinness conjecture of $S\setminus C$ in \cite{Marco:2010} by Marco Brunella, the remaining quasi-projective varieties we have to deal with are all in a kind of compactifiable deformation family $(\widetilde{\mathcal{F}}_1,\widetilde{\mathcal{S}}_3,V_i,\pi_{\widetilde{\mathcal{F}}_1},x_{origin},j_{origin})$ of $S\setminus C$ over a one-dimensional complex manifold constructed from the complex analytic family $(\widetilde{\mathcal{F}}_1,V_i,\pi_{\widetilde{\mathcal{F}}_1})$ for $i \in \{1,...,9\}$. The fibers for $(\widetilde{\mathcal{F}}_1,\widetilde{\mathcal{S}}_3,V_i,\pi_{\widetilde{\mathcal{F}}_1},x_{origin},j_{origin})$ are almost not Stein and the anticanonical line bundles for almost all the fibers of $(\widetilde{\mathcal{F}}_1,V_i,\pi_{\widetilde{\mathcal{F}}_1})$ are semipositive. A natural question is whether the anticanonical line bundle is semipositive for every fiber of the complex analytic family $(\widetilde{\mathcal{F}}_1,V_i,\pi_{\widetilde{\mathcal{F}}_1})$. The question corresponds to a conjecture proposed by Takayuki Koike \cite[Conjecture~2.1]{Takayuki:2023}.

At the end of this paper, we construct a deformation family of the canonical extension of a compact complex manifold.

Let $(\mathcal{M},\pi_\mathcal{M},B)$ be a nontrivial complex analytic family with B being a sufficiently small polydisc centered at 0. Here, $\pi_\mathcal{M}^{-1}(0)$ is holomorphically isomorphic to a compact Kähler manifold M.

\begin{proposition}
\label{pro1.5}
If there exists a polydisc $B_1 \Subset B$ such that $\mathcal{M}_{B_1}=\pi_\mathcal{M}^{-1}(B_1)$ is Kähler, then we can construct a deformation family of the canonical extension $W_M$ of the compact complex manifold M in the sense of \cite{Gasparim:2021} with the base manifold being $B_1$.
\end{proposition}

The condition above for the complex analytic family $(\mathcal{M},\pi_\mathcal{M},B)$ can be replaced by the other condition proposed by Jian Chen in \cite[Theorem~1.2]{Chen:2026}.

In addition, it was conjectured in \cite[Conjecture~1.1]{Höring:2024} that the fibers for the deformation family of the canonical extension $W_M$ of M constructed above are all Stein if and only if the tangent bundle for each fiber of the complex analytic family $(\mathcal{M}_{B_1},\pi_{B_1},B_1)$ is numerically effective. Some compact complex manifolds with numerically effective tangent bundles were studied in \cite{Schneider:1994} by Jean-Pierre Demailly, Thomas Peternell and Michael Schneider.

Here, one typical example is the deformation family of the canonical extension of a torus whose fibers are Stein.

The following is to introduce the main contents in the following sections.

In Section 2, we construct a nontrivial compactifiable deformation family of open Calabi-Yau manifolds and prove Theorem~\ref{thm1.1}.

Firstly, we present a complex analytic family $(\widetilde{\mathcal{F}},T,\widetilde{\pi})$ of rational elliptic surfaces with global sections and a complex analytic family of elliptic curves, $(\widetilde{\mathcal{S}},T,\widetilde{\pi}|_{\widetilde{\mathcal{S}}})$.

Secondly, through suitable choice of T, we can get a nontrivial deformation family of open Calabi-Yau manifolds, $(\widetilde{\mathcal{F}},\widetilde{\mathcal{S}},T,\widetilde{\pi},o,j)$.

Then Theorem~\ref{thm1.1} can be proved through combining Theorem~\ref{thm2.6} and Corollary~\ref{cor2.7}.

In Section 3, we give a condition for the existence of an elliptic fibration with global sections  on the complex analytic family to prove Corollary~\ref{cor1.2}, construct
a holomorphic involution to prove Corollary~\ref{cor1.3} and prove the non-Steinness of all the compactifiable deformation families constructed above.

In Section 4, we prove Proposition~\ref{pro1.4} related to the non-Steinness conjecture of $S\setminus C$ proposed by Marco Brunella in \cite{Marco:2010}, give some discussions on the remaining work of the conjecture and construct deformation families in the sense of \cite{Gasparim:2021} to prove Proposition~\ref{pro1.5}.

\section{A deformation family of open Calabi-Yau manifolds}

\noindent
This section is to construct a nontrivial compactifiable deformation of open Calabi-Yau manifolds. 

Firstly, we give a complex analytic family of rational elliptic surfaces with global sections. Taking use of the notations from the Introduction Section, the following lemma is true.

\begin{lemma}
\label{lem2.1}
Let $C^*=\mathbb{C}/\left \langle 1,\tau^* \right \rangle $   with $\tau^* \in Y$. Then for any two holomorphic embedding maps $f_1$ and $f_2$ from the elliptic curve $C^*$ to \emph{$\text{CP}^2$}, there exists an algebraic isomorphism g such that $f_1=g \circ f_2$.
\end{lemma}

\begin{proof}
$f_1 (C_0 (\tau^* ))$  and $f_2 (C_0 (\tau^* ))$ are two elliptic curves in $\text{CP}^2$ with the same j-invariant. Taking use of \cite[Chapter III-Proposition~3.1(a)]{Silverman:2009} and \cite[Chapter III-Proposition~3.1(b)]{Silverman:2009}, we can get an algebraic isomorphism between $C_1=f_1 (C_0 (\tau))$ and $C_2=f_2 (C_0 (\tau))$.That is to say, there exist regular rational maps $g:C_1\rightarrow C_2$ and $h:C_2 \rightarrow C_1$ such that $f_2 (C_0 (\tau))=g\circ f_1 (C_0 (\tau))$ and $f_1 (C_0 (\tau))=h \circ f_2 (C_0 (\tau))$. So Lemma~\ref{lem2.1} is proved.              
\end{proof}

Here, $f_\tau (C_0 (\tau))$ is an algebraic curve embedded in $\text{CP}^2$. In particular, there is a specific embedding map $F_\tau$ for $C_0 (\tau)$ embedded in $\text{CP}^2$ induced by the Weierstrass elliptic function defined in \cite[Introduction]{xu:2025}. 

Taking use of Lemma~\ref{lem2.1} above, we have an algebraic isomorphism between $f_\tau (C_0 (\tau_0))$ and $F_\tau(C_0 (\tau_0))$ to give the defining equation of $f_\tau (C_0 (\tau))$ as 
\begin{eqnarray}
P_1 (\tau) z^3+P_2 (\tau)xz^2+P_3 (\tau)yz^2+P_4 (\tau) x^2 z+P_5 (\tau)xyz+P_6 (\tau) x^3  \nonumber \\
+P_7 (\tau) x^2 y+P_8 (\tau)xy^2+P_9 (\tau) y^3+P_{10} (\tau)zy^2=0
\end{eqnarray}

\noindent
with $P_i (\tau)$ being holomorphic functions in a sufficiently small disc neighborhood $U \subset Y$  of $\tau_0$ for $i \in \{1,…,10\}$.

Now we have the following lemma.

\begin{lemma}
\label{lem2.2}
\cite[Main Theorem~2.1(1)]{Fujimoto:1990}: For any $\tau \in Y$, $f_\tau (C_0 (\tau))$ has a fixed inflexion point $Q_\tau$ and $f_\tau^{-1} (Q_\tau)$ is the zero point of $C_0 (\tau)$. So $S \cong \text{CP}^2 \text{$\#$9} \overline{\text{CP}}^2$ is a rational elliptic surface with global sections if and only if $\sum_{i=1}^9 f_\tau^{-1} (p_i)=0$ mod $\left \langle 1,\tau\right \rangle$.
\end{lemma}

Here, for any $\tau \in Y$, taking use of the construction of the defining equation (2.1) above, there is a fixed inflexion point $Q_\tau$ on the elliptic curve $f_\tau (C_0 (\tau))$. In addition, the natural group structure on $f_\tau (C_0 (\tau))$ coincides with the group structure on $\mathbb{C}/\left \langle 1,\tau\right \rangle$ through the holomorphic embedding map $f_\tau$ with $Q_\tau=f_\tau(0) $ being the identity. The proof of Lemma~\ref{lem2.2} is direct from \cite[Main Theorem~2.1(1)]{Fujimoto:1990}.

Now let $\pi$ be the natural projection mapping $\text{CP}^2 \times U$ to U. In addition, let $S_1=\left \{(x,\tau) \in \text{CP}^2 \times U|x \in f_\tau (C_0 (\tau))\right \}$.

\begin{lemma}
\label{lem2.3}
\emph{$(\text{CP}^2\times U,U,\pi)$} and $(S_1,U,\pi|_{S_1 })$ are complex analytic families.
\end{lemma}

Depending on the defining equation (2.1) of $f_\tau (C_0 (\tau))$ above with coefficients $P_i (\tau)$ being holomorphic for $i\in \{1,\cdots,10\}$, the proof of Lemma~\ref{lem2.3} is simple and similar to the proof of \cite[Lemma~3.2]{xu:2025}. 

Let $S^*=S_1 \times W_1 \times \cdots \times W_8$. Moreover, let $f_9 (\widetilde{t}) \in C_0 (\tau)$ be the point fixed by the equation
\[ \sum_{j=1}^8 f_{\tau,j}(\widetilde{p}_j)+f_9(\widetilde{t})=0 \quad  mod \quad \langle 1,\tau \rangle \]
for $\widetilde{t}=(\tau,\widetilde{p}_1,…,\widetilde{p}_8) \in T$. Define $\pi$ to be the natural projection mapping $\text{CP}^2 \times T$ to T. Then $(\text{CP}^2 \times T,T,\pi )$ is a complex analytic family of $\text{CP}^2$ and $(S^*,T,\pi|_{S^*}  )$ is a complex analytic family of elliptic curves.

Now we blow up nine points on each fiber of the complex analytic family $(\text{CP}^2 \times T,T,\pi )$. For any  $\widetilde{t}=(\tau,\widetilde{p}_1,…,\widetilde{p}_8) \in T$, we blow up nine points in the set
\[\{f_\tau (f_{\tau,1} (\widetilde{p}_1 )),…,f_\tau (f_{\tau,8} (\widetilde{p}_8 )),f_\tau (f_9 (\widetilde{t}))\}\times \{\widetilde{t}\} \subset \pi|_{S^*}^{-1} (\widetilde{t})\]
on the fiber $\pi^{-1} (\widetilde{t})\cong \text{CP}^2$.

Then we can get an 11-dimensional complex manifold $\widetilde{\mathcal{F}}$. Now let $\widetilde{\mathcal{S}} \subset \widetilde{\mathcal{F}}$ be the strict transform of $S^*$. In addition, let $\widetilde{\pi}:\widetilde{\mathcal{F}} \rightarrow T$ be a map induced by $\pi$. That is to say, for any $\widetilde{t} \in T$,  $\widetilde{\pi}^{-1} (\widetilde{t})\cong \text{CP}^2 \text{$\#$9} \overline{\text{CP}}^2$ is the blow-up of 
$\pi^{-1} (\widetilde{t})$ at nine points on $\pi|_{S^*}^{-1} (\widetilde{t})$. Then we have the following theorem.

\begin{theorem}
\label{thm2.4}
$(\widetilde{\mathcal{F}},T,\widetilde{\pi})$ is a complex analytic family of rational elliptic surfaces over the nine-dimensional complex manifold T and $(\widetilde{\mathcal{S}},T,\widetilde{\pi}|_{\widetilde{\mathcal{S}}})$ is a complex analytic family of elliptic curves over T.
\end{theorem}

\begin{proof}
Fixing a point $(x,y) \in S^* \subset \text{CP}^2 \times T$ with $y \in T$, $(x,y)$ can be deformed to a submanifold $M_1$ of $S^* \subset \text{CP}^2 \times T$.

Let $\Delta_{x,y} \subset \pi^{-1} (y) \cong \text{CP}^2$ be a disc neighborhood of $(x,y)$. The neighborhood $\Delta_{x,y}$ of $(x,y)$ in $\pi^{-1} (y)$ can be deformed to be a neighborhood $\Delta^*$ of $M_1$ in $\text{CP}^2 \times T$. Now for a point $(s,t) \in M_1$ with $t \in T$, let $l=[l_1:l_2:l_3]$ and $b=[b_1:b_2 ]$ be the homogeneous coordinates of $\text{CP}^2 \times \{t\}$ and $\text{CP}^1$ respectively.

Let $U_i=\{[x_1:x_2:x_3 ] \in \text{CP}^2 |x_i \neq 0\}$ for $i \in I=\{1,2,3\}$. Let $\varphi_1 ([x_1:x_2:x_3])=(x_2/x_1,$ $x_3/x_1 )$ for $ [x_1:x_2:x_3 ] \in U_1$, $\varphi_2 ([x_1:x_2:x_3])=(x_1/x_2 ,x_3/x_2 )$ for $[x_1:x_2:x_3 ] \in U_2$ and $\varphi_3 ([x_1:x_2:x_3])=(x_1/x_3 ,x_2/x_3 )$ for $[x_1:x_2:x_3 ] \in U_3$. Then $\{(U_i,\varphi_i)\}_{i\in I}$ is an atlas on $\text{CP}^2$. In addition, let $\phi_i (x,y)=(\varphi_i (x),y)$ for $y \in T$ and $x\in U_i$  with $i \in I$. Then $\{ (U_i \times T,\phi_i) \}_{i \in I}$ is an atlas of $\text{CP}^2 \times T$. 

Now there exists $j \in I$ such that $(s,t) \in U_j \times T$. Then let $\widetilde{\Delta}_{s,t} \subset \Delta_{s,t} \times \text{CP}^1$ be a submanifold of $\Delta_{s,t} \times \text{CP}^1$ defined by 
\begin{eqnarray}
\widetilde{\Delta}_{s,t} = \{(l,t,b) \in \Delta_{s,t} \times \text{CP}^1 | \varphi_j (l )=(l_{1,j}^*,l_{2,j}^* ), \varphi_j (s)=(s_1,s_2 ) \quad \& \nonumber \\
\forall \quad m \neq n \in \{1,2\}, \quad b_m (l_{n,j}^*-s_n)=b_n (l_{m,j}^*-s_m)  \} \nonumber
\end{eqnarray}

\noindent
with $\Delta_{s,t}=\Delta^* \cap \pi^{-1} (t)$. Now we have a projection $\pi_{s,t}:\widetilde{\Delta}_{s,t} \rightarrow \Delta_{s,t} \subset \pi^{-1} (t)$ defined by $\pi_{s,t}(l,t,b)=(l,t)$. Then we can define a blow-up $\widetilde{M}_{s,t}$ of $\text{CP}^2 \times \{t\} \cong \pi^{-1} (t)$ at $(s,t)$ as the complex manifold
\[\widetilde{M}_{s,t}=\pi^{-1}(t)-\{(s,t)\}\cup_{\pi_{s,t}} \widetilde{\Delta}_{s,t}\]
with the natural projection map $\widetilde{\pi}_{s,t}:\widetilde{M}_{s,t} \rightarrow \pi^{-1} (t)$. Here, $\widetilde{\pi}_{s,t}|_{\widetilde{\Delta}_{s,t}}=\pi_{s,t}$.

Now $\Delta=\cup_{(s,t) \in M_1}\widetilde{\Delta}_{s,t}$  is a submanifold of $\Delta^* \times \text{CP}^1$. Moreover, we have a projection $\pi_\Delta:\Delta \rightarrow \Delta^*$ defined by $\pi_\Delta|_{\widetilde{\Delta}_{s,t}}=\widetilde{\pi}_{s,t}|_{\widetilde{\Delta}_{s,t}}=\pi_{s,t}$.

Then we have a blow-up $\widetilde{M}_1$ of $\text{CP}^2 \times T$ along $M_1$ defined to be the complex manifold
\[\widetilde{M}_1=\text{CP}^2 \times T-M_1 \cup_{\pi_\Delta}\Delta\]
together with the natural projection map $\pi_{\widetilde{M}_1}:\widetilde{M}_1 \rightarrow \text{CP}^2 \times T$. Here, $\pi_{\widetilde{M}_1} |_\Delta=\pi_\Delta$.

Now let $\widetilde{\pi}_{\widetilde{M}_1}:\widetilde{M}_1 \rightarrow T$ be the map defined by $\widetilde{\pi}_{\widetilde{M}_1}=\pi \circ \pi_{\widetilde{M}_1 }$.

Let 
\[U_{i,j}=\{(l,t,d) \in \Delta|l \in U_i,d=(d_1,d_2 ) \quad with \quad d_j \neq 0,t \in T\}\]
for $j \in J=\{1,2\}$ and $i \in I$. For any $(i,j) \in I \times J$, we can define maps as below:
\[ \varphi_{i,j}^k (l,t,d)=\begin{cases}
    \frac{l_{k,i}^*-s_k}{l_{j,i}^*-s_j}=\frac{d_k}{d_j} \quad if \quad l_{j,i}^* \neq s_j\\
    \frac{d_k}{d_j} \quad  if  \quad l_{j,i}^*=s_j 
\end{cases},
\varphi_i (s)=(s_1,s_2) \quad with \]
\[\{(s,t)\}=M_1 \cap \pi^{-1} (t),  
 \varphi_i (l)=(l_{1,i}^*,l_{2,i}^* ) \quad for \quad k \neq j \quad \& \quad k \in \{1,2\},\]
\[\varphi_{i,j}^k (l,t,d)=l_{j,i}^*, \varphi_i (l)=(l_{1,i}^*,l_{2,i}^* ) \quad for \quad k=j,\]
\[  \varphi_{i,j}^k (l,t,d)=t_{k-2} \quad for \quad t=(t_1,…,t_9) \quad \&  \quad k \in \{3,…,11\} \]
with $(l,t,d) \in U_{i,j} \subset \Delta$, $l \in U_i$, $t \in T$ and $d=(d_1,d_2 )$ with $d_j \neq 0$. 

Then $\{(U_{i,j},\varphi_{i,j}=(\varphi_{i,j}^1,…,\varphi_{i,j}^{11}))\}_{(i,j) \in I \times J}$ is an atlas of the submanifold $\Delta$ of $\widetilde{M}_1$. Now $\widetilde{M}_1$ is an 11-dimensional complex manifold.

In addition, from the construction above, it is obvious that $\pi_{\widetilde{M}_1}|_{\widetilde{M}_1\setminus \pi_{\widetilde{M}_1}^{-1} (M_1)}$ is a biholomorphic map. Here, $\pi_{\widetilde{M}_1}$ and $\widetilde{\pi}_{\widetilde{M}_1}$ are holomorphic. On $ \widetilde{M}_1 \setminus \pi_{\widetilde{M}_1}^{-1} (M_1)$, the Jacobian matrix of $\widetilde{\pi}_{\widetilde{M}_1}$ is
\[\begin{pmatrix}
\begin{matrix}
    0&0\\
    \dots & \dots \\
    0&0 
\end{matrix}
\text{E}
\end{pmatrix}\]
with E being a 9×9 identity matrix.

Now define 
\[h_{i,j} (l,t,d)=\varphi_{i,j}^j (l,t,d)-[\varphi_i (s)]_j\] 
for  $(l,t,d) \in U_{i,j}$ with $l \in U_i$, $t \in T$, $\{(s,t)\}=M_1 \cap \pi^{-1} (t)$ and $\varphi_i(s)=(s_1,s_2)$. Then $\pi_{\widetilde{M}_1}^{-1} (M_1)\cap U_{i,j}$ is defined by the single equation $h_{i,j} (l,t,d)=0$. So $\pi_{\widetilde{M}_1}^{-1}(M_1)$ is a smooth hypersurface in $\widetilde{M}_1$.

Since the exceptional divisor $\pi_{\widetilde{M}_1}^{-1} (M_1)$ is just a hypersurface of $\widetilde{M}_1$, the Jacobian matrix of $\widetilde{\pi}_{\widetilde{M}_1}$ at every point of $\pi_{\widetilde{M}_1}^{-1} (M_1)$ is also
\[\begin{pmatrix}
\begin{matrix}
    0&0\\
    \dots & \dots \\
    0&0 
\end{matrix}
\text{E}
\end{pmatrix}\]
with E being the 9×9 identity matrix. This can also be computed through the atlas of the submanifold $\Delta$ given above. Now the rank for the Jacobian matrix of $\widetilde{\pi}_{\widetilde{M}_1}$ is 9 at every point of $\widetilde{M}_1$.

Blowing up $\widetilde{M}_i$ along the remaining eight submanifolds $M_{i+1}$ traced by the blow-up points corresponding to the set T one by one with the similar construction method as above for $i\in \{1,…,8\}$, we can get the manifold $\widetilde{M}_i$ for $i\in \{2,…,9\}$ with $\widetilde{M}_9 = \widetilde{\mathcal{F}}$. Then we can easily get the fact that $\widetilde{\pi}:\widetilde{\mathcal{F}} \rightarrow T$ is holomorphic and the rank for the Jacobian matrix of the map is 9 at every point of $\widetilde{\mathcal{F}}$. It is obvious that $\widetilde{\pi}^{-1}(y)\cong \text{CP}^2 \text{$\#$9} \overline{\text{CP}}^2$ is a compact complex submanifold of $\widetilde{\mathcal{F}}$ for any $y \in T$. So $(\widetilde{\mathcal{F}},T,\widetilde{\mathcal{\pi}})$ is a complex analytic family of $\text{CP}^2 \text{$\#$9} \overline{\text{CP}}^2$.

$\widetilde{S}$ which is a complex submanifold of $\widetilde{\mathcal{F}}$ is the strict transform of $S^*$ . So $\widetilde{\pi}|_{\widetilde{S}}$  is a holomorphic map with the rank of the Jacobian matrix being 9 at every point of $\widetilde{S}$.

In addition, $\forall x \in T$, $\widetilde{\pi}|_{\widetilde{S}}^{-1}(x)$ is the strict transform of $\pi|_{S^*}^{-1}(x)$ in $\widetilde{\pi}^{-1}(x)$. So $\widetilde{\pi}|_{\widetilde{S}}^{-1}(x)$ is a compact complex submanifold of $\widetilde{S}$. Then $(\widetilde{S},T,\widetilde{\pi}|_{\widetilde{S}})$ is a complex analytic family. So Theorem~\ref{thm2.4} is proved.  
\end{proof}

Moreover, we can have a corollary as follows:

\begin{corollary}
\label{cor2.5}
The blow-up of \emph{$\text{CP}^2 \times T$} along the submanifold $M_1$ constructed above is just the blow-up of \emph{$\text{CP}^2 \times T$} centered at $M_1$ defined in \cite[Chapter VII-12.3]{Demailly:1997}.
\end{corollary}

\begin{proof}
From the proof of Theorem~\ref{thm2.4}, $\widetilde{M}_1$ is an 11-dimensional complex manifold and $\pi_{\widetilde{M}_1}$ is a holomorphic map. $\pi_{\widetilde{M}_1}|_{\widetilde{M}_1 \setminus \pi_{\widetilde{M}_1}^{-1} (M_1)} $ is a biholomorphic map and $\pi_{\widetilde{M}_1}^{-1}(M_1)$ is a smooth hypersurface in $\widetilde{M}_1$.

We can construct the normal bundle $N_{M_1/{\text{CP}^2 \times T}}$ of $M_1$ in $\text{CP}^2 \times T$ as the collection of the holomorphic tangent spaces $T^{'}_{(s,t)}{(\text{CP}^2 \times \{t\})}$ for $(s,t) \in M_1$.

For any $i \in I$, $(x,t) \in U_i \times T$ and $(l_t,t)\in (U_i \times T) \cap M_1$ with $t \in T$, let $(l_{1,t,i}^*,l_{2,t,i}^* )=\varphi_i(l_t)$ and $(x_{1,i},x_{2,i} )=\varphi_i(x)$. Then the vector fields in the set 
\[\{\partial/\partial x_{1,i}|_{x_{1,i}=l_{1,t,i}^*},\partial/\partial x_{2,i}|_{x_{2,i}=l_{2,t,i}^*}\}\]
\noindent
provide a holomorphic frame of $N_{M_1/\text{CP}^2 \times T}|_{(U_i\times T)\cap M_1 }$.

Define 
\[h_i(l_t,t,b_{1,i} \partial/\partial x_{1,i}|_{x_{1,i}=l_{1,t,i}^*}+b_{2,i} \partial/\partial x_{2,i}|_{x_{2,i}=l_{2,t,i}^*})=(t,b_{1,i},b_{2,i} )\]
for $(l_t,t,b_{1,i} \partial/\partial x_{1,i}|_{x_{1,i}=l_{1,t,i}^*}+b_{2,i} \partial/\partial x_{2,i}|_{x_{2,i}=l_{2,t,i}^*}) \in N_{M_1/\text{CP}^2 \times T}|_{(U_i \times T)\cap M_1}$. 

Then $\{(N_{M_1/\text{CP}^2 \times T} |_{(U_i \times T) \cap M_1 },h_i) \}_{i \in I}$ is an atlas on the total space $N_{M_1/\text{CP}^2 \times T}$.

For $j\in J$, let
\[U_{i,j}^*=\{(l_t,t,P_{l_t,t} (b_{1,i} \partial/\partial x_{1,i}|_{x_{1,i}=l_{1,t,i}^*}+b_{2,i} \partial/\partial x_{2,i}|_{x_{2,i}=l_{2,t,i}^*}))\]
\[\in P(N_{M_1/\text{CP}^2 \times T})|_{(U_i \times T)\cap M_1} | b_{j,i} \neq 0\}\]
with 
\[P_{l_t,t} (b_{1,i} \partial/\partial x_{1,i}|_{x_{1,i}=l_{1,t,i}^*}+b_{2,i} \partial/\partial x_{2,i}|_{x_{2,i}=l_{2,t,i}^*} )\]
\[=(b_{1,i} \partial/\partial x_{1,i}|_{x_{1,i}=l_{1,i}^*}+b_{2,i} \partial/\partial x_{2,i}|_{x_{2,i}=l_{2,t,i}^*})/\sim\]
where $\sim$ is the equivalence relation generated by
\[b_{1,i} \partial/\partial x_{1,i}|_{x_{1,i}=l_{1,t,i}^*}+b_{2,i} \partial/\partial x_{2,i}|_{x_{2,i}=l_{2,t,i}^*} \sim \lambda (b_{1,i}^* \partial/\partial x_{1,i}|_{x_{1,i}=l_{1,t,i}^*}+b_{2,i}^* \partial/\partial x_{2,i}|_{x_{2,i}=l_{2,t,i}^*})\]
for $\lambda \in \mathbb{C}\setminus \{0\}$ and $\{(b_{1,i},b_{2,i} ),(b_{1,i}^*,b_{2,i}^* )\} \subset \mathbb{C}^2$ with $b_{j,i} b_{j,i}^* \neq 0$.

Let \[\widetilde{h}_{i,j} (l_t,t,P_{l_t,t} (b_{1,i} \partial/\partial x_{1,i}|_{x_{1,i}=l_{1,t,i}^*}+b_{2,i} \partial/\partial x_{2,i}|_{x_{2,i}=l_{2,t,i}^*}))=(t,b_{m,i}/b_{j,i})\]
for $m \neq j \in \{1,2\}$ and $(l_t,t,P_{l,t} (b_{1,i} \partial/\partial x_{1,i}|_{x_{1,i}=l_{1,t,i}^*}+b_{2,i} \partial/\partial x_{2,i}|_{x_{2,i}=l_{2,t,i}^*})) \in U_{i,j}^*$.

Now $\{(U_{i,j}^*, \widetilde{h}_{i,j})\}_{i \in I,j \in J}$ is an atlas on the total space $P(N_{M_1/\text{CP}^2 \times T})$. For any $(i,j)\in I \times J$, let 
\[f_N (l_t,t,d)=(l_t,t,P_{l_t,t} (d_1 \partial/\partial  x_{1,t,i} |_{ x_{1,t,i}=l_{1,t,i}^*}+d_2 \partial/\partial x_{2,t,i} |_{ x_{2,t,i}=l_{2,t,i}^*})) \in U_{i,j}^*\]
for $(l_t,t,d)\in U_{i,j} \cap \pi_{\widetilde{M}_1}^{-1} (M_1 )$ with $l_t \in U_i$, $t\in T$ and $d=(d_1,d_2 )$. Then $f_N:\pi_{\widetilde{M}_1}^{-1} (M_1 ) \rightarrow P(N_{M_1/\text{CP}^2 \times T})$ is a globally well-defined map and it is not hard to see that $f_N$ is biholomorphic. Therefore, $\pi_{\widetilde{M}_1}^{-1} (M_1)$ is holomorphically isomorphic to the projectivized normal bundle $P(N_{M_1/\text{CP}^2 \times T})$.

In conclusion, the blow-up of $\text{CP}^2 \times T$ along the submanifold $M_1$ constructed above is just the blow-up of $\text{CP}^2 \times T$ centered at $M_1$ defined in  \cite[Chapter VII-12.3]{Demailly:1997}. So Corollary~\ref{cor2.5} is proved.
\end{proof}

With the same proof as Corollary~\ref{cor2.5}, we can prove that $\widetilde{\mathcal{F}}$ is exactly the manifold generated through blowing up $\widetilde{M}_i$ centered at the submanifolds $M_{i+1}$ traced by the blow-up points corresponding to T for $i\in \{1,…,8\}$ one by one. Moreover, we give the following definition.

\begin{definition}
\label{def1}
If every two distinct fibers of the compactifiable deformation family are algebraically isomorphic to each other, then we say that the compactifiable deformation family is algebraically trivial. Simply, we call it a trivial deformation family. Otherwise, it can be called a nontrivial deformation family. 
\end{definition}

Now we can give a deformation family of open Calabi-Yau manifolds similar to the deformation family constructed in \cite[Section~3]{xu:2025}. Taking use of the notations from the Introduction Section, the following theorem is true.

\begin{theorem}
\label{thm2.6}
Choosing T properly, $(\widetilde{\mathcal{F}},\widetilde{\mathcal{S}},T,\widetilde{\pi},o,j)$ is a compactifiable deformation of the open Calabi-Yau manifold $S_0\setminus C_0$.
\end{theorem}

\begin{proof}
Firstly, $S_0\setminus C_0$ is compactifiable.

The proof is similar to the proof of \cite[Theorem~3.13]{xu:2025}. It is not so hard to see that $j(S_0\setminus C_0)=(\widetilde{\pi})^{-1}(o)\setminus \widetilde{S}$ with $(\widetilde{\pi})^{-1}(o)$ being a compact complex space and $(\widetilde{\pi})^{-1}(o)\cap \widetilde{S}$ being a closed analytic subset. Moreover, j is an open embedding. So $S_0\setminus C_0$ is a compactifiable complex space.

Secondly, $\widetilde{S}$ is a closed analytic subset of $\widetilde{\mathcal{F}}$ and $\widetilde{\pi}:\widetilde{\mathcal{F}} \rightarrow T$ is a proper holomorphic map with $\widetilde{\pi}|_{\widetilde{\mathcal{F}} \setminus \widetilde{\mathcal{S}}}$ being flat.

Here, the complex manifolds $\widetilde{\mathcal{F}}$ and T are complex spaces. Similar to the proof of \cite[Theorem~3.13]{xu:2025}, we can get an open neighborhood for every point in $\widetilde{\mathcal{F}} \setminus \widetilde{\mathcal{S}}$ easily. So $\widetilde{\mathcal{S}}$ is a closed analytic subset of $\widetilde{\mathcal{F}}$.

Moreover, for the complex analytic family $(\widetilde{\mathcal{F}},T,\widetilde{\pi})$, $\widetilde{\pi}:\widetilde{\mathcal{F}} \rightarrow T$ is a proper submersion. Then $\widetilde{\pi}|_{\widetilde{\mathcal{F}} \setminus \widetilde{\mathcal{S}}}$ is flat.

Now we can get the fact that $(\widetilde{\mathcal{F}},\widetilde{\mathcal{S}},T,\widetilde{\pi},o,j)$ is a compactifiable deformation of $S_0\setminus C_0$.

Finally, Hans-Joachim Hein constructed some complete Calabi-Yau metrics on the quasi-projective variety $S\setminus C$ in  \cite{Hein:2012}. Therefore, the fibers of the deformation family $(\widetilde{\mathcal{F}},\widetilde{\mathcal{S}},T,\widetilde{\pi},o,j)$ are all complete open Calabi-Yau manifolds.

In conclusion, $(\widetilde{\mathcal{F}},\widetilde{\mathcal{S}},T,\widetilde{\pi},o,j)$ is a compactifiable deformation of the open Calabi-Yau manifold $S_0\setminus C_0$. So Theorem~\ref{thm2.6} is proved.                          
\end{proof}

In addition, taking use of the same proof as \cite[Proposition~3.14]{xu:2025}, it is not hard to see that $(\widetilde{\mathcal{F}},\widetilde{\mathcal{S}},T,\widetilde{\pi},o,j)$ is a smooth compactifiable deformation differentially trivial along T. The following is a proposition for the non-triviality of the deformation.

\begin{corollary}
\label{cor2.7}
Choosing T properly, $(\widetilde{\mathcal{F}},\widetilde{\mathcal{S}},T,\widetilde{\pi},o,j)$ is a nontrivial deformation family. 
\end{corollary}

The proof of Corollary~\ref{cor2.7} is similar to the statement in \cite[Introduction]{xu:2025}. Choosing T properly, for any $\tau_1 \neq \tau_2 \in U$ and $(\tau_i,p_1,...,p_8) \in T$ with $i \in \{1,2\}$, the square-zero elliptic curves $\widetilde{\pi}|_{\widetilde{\mathcal{S}}}^{-1}(\tau_1,p_1,...,p_8)$ and $\widetilde{\pi}|_{\widetilde{\mathcal{S}}}^{-1}(\tau_2,p_1,...,p_8)$ are not isomorphic to each other. Then there is not any algebraic isomorphism between $\widetilde{\pi}^{-1}(\tau_1,p_1,...,p_8)\setminus \widetilde{\pi}|_{\widetilde{\mathcal{S}}}^{-1}(\tau_1,p_1,...,p_8)$ and $\widetilde{\pi}^{-1}(\tau_2,p_1,...,p_8)\setminus \widetilde{\pi}|_{\widetilde{\mathcal{S}}}^{-1}(\tau_2,p_1,...,p_8)$. So $(\widetilde{\mathcal{F}},\widetilde{\mathcal{S}},T,\widetilde{\pi},o,j)$ is a nontrivial deformation family of open Calabi-Yau manifolds.

Now combining Theorem~\ref{thm2.6} and Corollary~\ref{cor2.7}, the proof of Theorem~\ref{thm1.1} is completed.

\section{A deformation family admitting a holomorphic involution and non-Steinness}

\noindent
In this section, we give a condition for the existence of an elliptic fibration with global sections on the complex analytic family $(\widetilde{\mathcal{F}},T,\widetilde{\pi})$, construct a holomorphic involution on the  compactifiable deformation family and claim the non-Steinness for all the compactifiable deformation families constructed.

Firstly, we present a condition for the existence of an elliptic fibration with global sections on the complex analytic family $(\widetilde{\mathcal{F}},T,\widetilde{\pi})$ and construct a holomorphic involution on the  compactifiable deformation family.

Let $f_\tau$ be defined as the embedding map constructed from the Weierstrass elliptic function $\wp$ introduced in \cite[Introduction]{xu:2025} for $\tau \in Y$. Let $p=(p_1,…,p_9)$ and $p^*=(p_1^*,...,p_9^*)=(f_{\tau,1}(p_1),...,f_{\tau,8}(p_8),f_{\tau,9}(p_9))$ with $\sum_{i=1}^9p^*_i=0$.

Let $\{r_1,r_2,r_3 \}$ be a basis of $H^0 (C_0,\mathcal{O}(3f_{\tau}^{-1} (Q)))$ such that $\wp(x)=r_1 (x)/r_3 (x)$ and $\wp^\prime(x) = r_2 (x)/r_3 (x)$ for $0 \neq x \in C_0$. 

Here, $r_3=\theta^3$ with $\theta(x) \in H^0 (C_0,\mathcal{O}(f_{\tau}^{-1}(Q)))$ being the first Jacobi Theta function which is the Riemann-theta function of degree one with an order one zero at the zero point. Let 
\[h_p (x):=\prod_{i=1}^9 \theta(x-p_i^*) \in H^0 (C_0,\mathcal{O}(9f_{\tau}^{-1}(Q)) ).\]

Then $\{r_3^3,r_1 r_3^2,r_2 r_3^2,r_1^2 r_3,r_1 r_2 r_3,r_1^3,r_1^2 r_2,r_1 r_2^2,r_2^3 \}$ form a basis of $H^0 (C_0,\mathcal{O}(9f_{\tau}^{-1}(Q)))$ with $f_{\tau}$ mapping the elements in 
\[\{b_{\tau,1}=r_3^3,b_{\tau,2}=r_1 r_3^2,b_{\tau,3}=r_2 r_3^2,b_{\tau,4}=r_1^2 r_3,b_{\tau,5}=r_1 r_2 r_3,b_{\tau,6}=r_1^3,b_{\tau,7}=r_1^2 r_2,b_{\tau,8}=r_1 r_2^2,b_{\tau,9}=r_2^3\}\]
to the elements in $\{z^3,xz^2,yz^2,x^2 z,xyz,x^3,x^2 y,xy^2,y^3\}$ respectively \cite{Fujimoto:1990}.

Now there exist nine elements in the set $\{Q_{\tau,1},\cdots,Q_{\tau,9}\} \subset  \mathbb{C}^9$ such that
\begin{eqnarray}
h_p (x)=\sum_{i=1}^9 Q_{\tau,i} b_{\tau,i} (x). 
\end{eqnarray}

Substituting the nine points in $\{p_1^*,\cdots,p_8^*,0\}$ into the equation (3.1), we can get a equation system consisting of $\sum_{i=1}^9 Q_{\tau,i} b_{\tau,i}(p_j^*)=0$ for $j\in \{1,\cdots,8\}$ and $Q_{\tau,9}=\prod_{i=1}^9\theta(-p_j^*)$.

The coefficient matrix of the equation system is
\[D=\begin{pmatrix}
b_{\tau,1}(p_1^*) & b_{\tau,2}(p_1^*) & \cdots & b_{\tau,9}(p_1^*) \\
\cdots & \cdots & \cdots & \cdots \\
b_{\tau,1}(p_8^*) & b_{\tau,2}(p_8^*) & \cdots & b_{\tau,9}(p_8^*) \\
0 & 0 & \cdots & 1 \\
\end{pmatrix}.\]

Here, let $M^1=\{(p_1,\cdots,p_9)|\{p_1,\cdots,p_9 \} \in \mathbb{C}, rank(D) \geq 8 \ \& \ \sum_{i=1}^9 f_{\tau,i}(p_i)=0 \}$.

Let $M^2=\{(p_1,\cdots,p_9)|\{p_1,\cdots,p_9 \} \in \mathbb{C}, rank(D)=9 \ \& \ \sum_{i=1}^9 f_{\tau,i}(p_i)=0 \}$ and $M^3=\{(p_1,\cdots,p_9)|\{p_1,\cdots,p_9 \} \in \mathbb{C}, \forall \ i \in \{1,\cdots, 9 \}, f_{\tau,i}(p_i) \neq 0 \}$.

For $p\in M^2 \cap M^3 \subset M^1$, since $rank(D)=9$, we can get a nonzero solution $Q_{\tau,1},\cdots,Q_{\tau,9}$ of the equation system above.

Let $M^4=\{(p_1,\cdots,p_9)|\{ p_1,\cdots,p_9\} \in \mathbb{C}, rank(D)=8 \ \& \ \sum_{i=1}^9f_{\tau,i}(p_i)=0 \}$. For $p \in M^4 \setminus M^3 \subset M^1$, since $rank(D)=8$, then we can get a one-dimensional solution space. For any nonzero solution $Q_{\tau,1}=a_1^*,\cdots,Q_{\tau,9}=a_9^*$ of the equation system above, the value $[a_1^*:\cdots:a_9^*]\in \text{CP}^8$ is unique.

Then for $p\in (M^2 \cap M^3)\cup(M^4 \setminus M^3) \subset M^1$, we get a function $P_\tau:(M^2 \cap M^3)\cup(M^4 \setminus M^3) \rightarrow \text{CP}^8$ defined by $P_\tau(p)=[Q_{\tau,1}:\cdots:Q_{\tau,9}]\in \text{CP}^8$ with $Q_{\tau,1},\cdots,Q_{\tau,9}$ being a nonezero solution for the system of equations above.

Moreover, we can get a cubic curve in $\text{CP}^2$ defined as follows.
\begin{eqnarray}C_{\tau,p}:g_{\tau,p}=Q_{\tau,1}z^3+Q_{\tau,2}xz^2+Q_{\tau,3}yz^2+Q_{\tau,4}x^2z+Q_{\tau,5}xyz+Q_{\tau,6}x^3  \nonumber \\
+Q_{\tau,7}x^2y+Q_{\tau,8}xy^2+Q_{\tau,9}y^3=0 \nonumber
\end{eqnarray}
\noindent
with $C_{\tau,p}\cap f_\tau(C_0(\tau))=\{ f_\tau(p^*_1),\cdots,f_\tau(p^*_9)\}$.

Now let $W=\{(p_1,\cdots,p_9)|\{ p_1,\cdots,p_9\}\in \mathbb{C}, (p_1,\cdots,p_8)=W_1\times\cdots \times W_8 \ \& \ \sum_{i=1}^9f_{\tau,i}(p_i)=0 \}$. Then we have the following corollary.

\begin{corollary}
\label{cor3.1}
If T is chosen properly such that $P_\tau$ is a holomorphic function on $W \subset (M^2 \cap M^3)\cup(M^4 \setminus M^3)$, then there exists an elliptic fibration with global sections on the complex analytic family $(\widetilde{\mathcal{F}},T,\widetilde{\pi})$.
\end{corollary}

\begin{proof}
Here, the family $t\cdot g_{\tau,p}+s \cdot f_{\tau,\infty}=0$, $[s:t] \in \text{CP}^1$, gives a pencil of cubics on the fiber $\widetilde{\pi}^{-1}(\tau,p_1,\cdots,p_8)$ of the complex analytic family $(\widetilde{\mathcal{F}},T, \widetilde{\mathcal{\pi}})$ for $(\tau,p_1,\cdots,p_8)\in T$. So we can get a unique elliptic fibration $f_{\tau,p_1,\cdots,p_8}$ on $\widetilde{\pi}^{-1}(\tau,p_1,\cdots,p_8)$ mapping the curve $t\cdot g_{\tau,p}+s \cdot f_{\tau,\infty}=0$ to the point $(\tau,p_1,\cdots,p_8,[s:t]) \in T \times \text{CP}^1$. Here, the elliptic fibration is unique up to a Möbius transformation on $\text{CP}^1$.

Now from the construction above, $t\cdot g_{\tau,p}+s \cdot f_{\tau,\infty}=0$ depends holomorphically on $\tau \in U$. So $f_{\tau,p_1,\cdots,p_8}$ depends holomorphically on $\tau \in U$. Let $f_{\widetilde{\mathcal{F}}}:\widetilde{\mathcal{F}} \rightarrow T \times \text{CP}^1$ be defined by $f_{\widetilde{\mathcal{F}}}|_{\widetilde{\pi}^{-1}(\tau,p_1,\cdots,p_8)}=f_{\tau,p_1,\cdots,p_8}$ for $(\tau,p_1,\cdots,p_8) \in T$. Then $f_{\widetilde{\mathcal{F}}}$ gives an elliptic fibration on $\widetilde{\mathcal{F}}$.

In addition, for $i\in \{1,...,9\}$, since the restriction of the projectivized normal bundle of $M_i$ on $\widetilde{\pi}^{-1}(\tau,p_1,\cdots,p_8)$ is an exceptional divisor, it gives a global holomorphic section of the elliptic fibration $f_{\tau,p_1,\cdots,p_8}$  for $(\tau,p_1,\cdots,p_8) \in T$. So the projectivized normal bundle of $M_i$ which is an exceptional divisor in $\widetilde{\mathcal{F}}$ corresponding to $M_i$ is a global holomorphic section of the elliptic fibration $f_{\widetilde{\mathcal{F}}}$. So Corollary~\ref{cor3.1} is proved.
\end{proof}

Moreover, we have a corollary as follows.

\begin{corollary}
\label{cor3.2}
For the deformation family constructed in Corollary~\ref{cor3.1}, if all the singular fibers are of Kodaira type $I_1$ and $C_{\tau,p}$ is a smooth elliptic curve for $p \in W$ and $\tau \in U$, then there exists a holomorphic involution on the compactifiable deformation family $(\widetilde{\mathcal{F}},\widetilde{\mathcal{S}},T,\widetilde{\pi},o,j)$ induced by the hypersurface involution on the rational elliptic surface with global sections.
\end{corollary}

\begin{proof}
Now we can construct a map $f_{inv}$ on $\widetilde{\mathcal{F}}$ with $f_{inv}|_{\widetilde{\pi}^{-1}(\tau,p_1,\cdots,p_8)}$ being a hypersurface involution on $\widetilde{\pi}^{-1}(\tau,p_1,\cdots,p_8)$ mapping every element to its inverse for $(\tau,p_1,\cdots,p_8) \in T$. Here, the group structure on each fiber of $f_{\widetilde{\mathcal{F}}}$ can be identified through the defining equations constructed in the proof of Corollary~\ref{cor3.1}.

It is obvious that the hypersurface involution $f_{inv}$ is holomorphic in a sufficiently small neighborhood of a smooth fiber. 

Here, let $\mathcal{M}_{1,1}$ be the orbifold constructed in \cite[3.5]{Hain:2008}. There is a global holomorphic involution defined on the universal curve $\mathcal{E} \rightarrow \mathcal{M}_{1,1}$ induced by the inverse map on the elliptic curve. In addition, the involution can be extened to the universal stable elliptic curve $\overline{\mathcal{E}} \rightarrow \overline{\mathcal{M}}_{1,1}$ with $\overline{\mathcal{M}}_{1,1}$ being the compactification of $\mathcal{M}_{1,1}$. Now since every singular fiber is of Kodaira type $I_1$ which is stable, the hypersurface involution $f_{inv}$ is holomorphic in a sufficiently small neighborhood of a singular fiber. 

Then $f_{inv}$ is a holomorphic involution on $\widetilde{\mathcal{F}}$ and $f_{inv}|_{\widetilde{\mathcal{F}}\setminus \widetilde{\mathcal{S}}}$ is a holomorphic involution on the compactifiable deformation family $(\widetilde{\mathcal{F}},\widetilde{\mathcal{S}},T,\widetilde{\pi},o,j)$. Then Corollary~\ref{cor3.2} is proved.
\end{proof}

Here, there is a fact that a generic rational elliptic surface with global sections only has singular fibers of Kodaira type $I_1$. 

Secondly, we give some useful examples here.

Taking use of the result from \cite[Example~3.3]{Fujimoto:1990}, Let $det(D)=G_{\tau,9}$ and $Q_{\tau,j}=(-1)^j\prod_{i=1}^9\theta(p^*_j)\frac{G_{\tau,j}}{G_{\tau,9}}$ for $j\in \{ 1,\cdots,9\}$. Now assume that $G_{\tau,i} \in H^0(M^2,\mathcal{O}(L))$ for $i \in \{ 1,\cdots, 9\}$ with $L$ being a line bundle on $M^2$. Let $M^0=\{(p_1,...,p_9) \in M^2|G_{\tau,i} \ and \ G_{\tau,j}$ are relatively prime for $ i \neq j \in \{1,\cdots,9\}\}\cap M^3$. Then $M^0$ is a submanifold of the manifold constructed in  \cite[Example~3.3]{Fujimoto:1990}.

Taking use of the result from \cite[Example~3.3]{Fujimoto:1990}, for a sufficiently small T satisfying the condition that $W \subset M^0$, $Q_{\tau,j}$ is a holomorphic function for $i\in \{1,\cdots,9\}$ which indicates that $P_\tau$ is a holomorphic function on $W$. Taking use of Corollay~\ref{cor3.1}, we can get an elliptic fibration with global sections on the complex analytic family $(\widetilde{\mathcal{F}},T,\widetilde{\pi})$.

Now if there are only singular fibers of Kodaira type $I_1$ in the deformation family and $C_{\tau,p}$ is a smooth elliptic curve for $p \in W$ and $\tau \in U$, taking use of Corollay~\ref{cor3.2}, then the compactifiable deformation family $(\widetilde{\mathcal{F}},\widetilde{\mathcal{S}},T,\widetilde{\pi},o,j)$ admits a holomorphic involution induced by the hypersurface involution on rational elliptic surface with global sections. 

The following is to give an example for which the existence of the elliptic fibration with global sections on the complex analytic family and the analogous holomorphic involution map can not be derived directly.

Here, the center point of T plays a key role for the constructed above. Taking use of the notations above, we give some discussions on the situation that the center point $o=(\tau_0,q_1,\cdots,q_8)$ of T satisfies the condition that $q_i=-q_{10-i}$ for $i \in \{2,\cdots,5\}$ in the following.

Let $q=(q_1,\cdots,q_9)$ and $q^*=(q_1^*,\cdots,q_9^*)=(f_{\tau,1}(q_1),...,f_{\tau,9}(q_9))$ with $\sum_{i=1}^9q^*_i=0$.

Now there exist a set $\{R_{\tau,1},\cdots,R_{\tau,9}\} \subset  \mathbb{C}^9$ such that
\begin{eqnarray}
h_q (x)=\sum_{i=1}^9 R_{\tau,i} b_{\tau,i} (x). 
\end{eqnarray}

Substituting the nine points in $\{q_1^*,\cdots,q_9^*\}$ into the equation (3.2), we can get a equation system consisting of $\sum_{i=1}^9 R_{\tau,i} b_{\tau,i}(q_j^*)=0$ for $j\in \{1,\cdots,9\}$.

Here, the Weierstrass $\wp$-function satisfies the properties as follows:
\[\wp(x)=\wp(-x) \quad \& \quad \wp^\prime(-x)=-\wp^\prime(x)\]
for $x \in C_0 \setminus \{0\}$.

The first Jacobi Theta function satisfies a property as follows:
\[\theta(x)=-\theta(-x)\]
for $x \in C_0$.

Then $\forall x \in C_0\setminus \{0\}$, $b_{\tau,i} (x)=-b_{\tau,i} (-x)$ for $i\in \{1,2,4,6,8\}$ and $b_{\tau,i} (x)=b_{\tau,i}(-x)$ for $i\in \{3,5,7,9\}$. So through suitable change, we can get a equation system consisting of nine equations as follows:
\begin{eqnarray}
\begin{cases}
R_{\tau,3} b_{\tau,3} (q_1^*) +R_{\tau,5} b_{\tau,5} (q_1^*) +R_{\tau,7} b_{\tau,7} (q_1^*) =0 \\
R_{\tau,3} b_{\tau,3} (q_2^*)+R_{\tau,5} b_{\tau,5} (q_2^*)+R_{\tau,7} b_{\tau,7} (q_2^*)=0 \\
R_{\tau,3} b_{\tau,3} (q_3^*) +R_{\tau,5} b_{\tau,5} (q_3^*) +R_{\tau,7} b_{\tau,7} (q_3^*) =0 \\
R_{\tau,3} b_{\tau,3} (q_4^*)+R_{\tau,5} b_{\tau,5} (q_4^*)+R_{\tau,7} b_{\tau,7} (q_4^*)=0 \\
R_{\tau,9}=0 \\
R_{\tau,1} b_{\tau,1} (q_6^*)+R_{\tau,2} b_{\tau,2} (q_6^*)+R_{\tau,4} b_{\tau,4} (q_6^*)+R_{\tau,6} b_{\tau,6} (q_6^*)+R_{\tau,8} b_{\tau,8} (q_6^*)=0 \\
R_{\tau,1} b_{\tau,1} (q_7^*)+R_{\tau,2} b_{\tau,2} (q_7^*)+R_{\tau,4} b_{\tau,4} (q_7^*)+R_{\tau,6} b_{\tau,6} (q_7^*)+R_{\tau,8} b_{\tau,8} (q_7^*)=0\\
R_{\tau,1} b_{\tau,1} (q_8^*)+R_{\tau,2} b_{\tau,2} (q_8^*)+R_{\tau,4} b_{\tau,4} (q_8^*)+R_{\tau,6} b_{\tau,6} (q_8^*)+R_{\tau,8} b_{\tau,8} (q_8^*)=0 \\
R_{\tau,1} b_{\tau,1} (q_9^*)+R_{\tau,2} b_{\tau,2} (q_9^*)+R_{\tau,4} b_{\tau,4} (q_9^*)+R_{\tau,6} b_{\tau,6} (q_9^*)+R_{\tau,8} b_{\tau,8} (q_9^*)=0
\end{cases}.
\end{eqnarray}

The coefficient matrix is
\[D=\begin{pmatrix}
0&0&b_{\tau,3} (q_1^*)&0&b_{\tau,5} (q_1^*)&0&b_{\tau,7} (q_1^*)&0&0 \\
   0&0&b_{\tau,3} (q_2^*)&0&b_{\tau,5} (q_2^*)&0&b_{\tau,7} (q_2^*)&0&0 \\
   0&0&b_{\tau,3} (q_3^*)&0&b_{\tau,5} (q_3^*)&0&b_{\tau,7} (q_3^*)&0&0 \\
   0&0&b_{\tau,3} (q_4^*)&0&b_{\tau,5} (q_4^*)&0&b_{\tau,7} (q_4^*)&0&0 \\
  0&0&0&0&0&0&0&0&1  \\
  b_{\tau,1}(q_6^*)&b_{\tau,2}(q_6^*)&0&b_{\tau,4}(q_6^*)&0&b_{\tau,6}(q_6^*)&0&b_{\tau,8}(q_6^*)&0   \\
  b_{\tau,1} (q_7^*)&b_{\tau,2} (q_7^*)&0&b_{\tau,4} (q_7^*)&0&b_{\tau,6} (q_7^*)&0&b_{\tau,8} (q_7^*)&0  \\
  b_{\tau,1} (q_8^*)&b_{\tau,2} (q_8^*)&0&b_{\tau,4} (q_8^*)&0&b_{\tau,6} (q_8^*)&0&b_{\tau,8} (q_8^*)&0  \\
  b_{\tau,1} (q_9^*)&b_{\tau,2} (q_9^*)&0&b_{\tau,4} (q_9^*)&0&b_{\tau,6} (q_9^*)&0&b_{\tau,8} (q_9^*)&0  \\
\end{pmatrix}.\]
\noindent
Here, it is obvious that rank(D) $\leq$ 8. Let
\[M^5=\{(q_1,\cdots,q_9 )\in \mathbb{C}^9 | \{ q_1,\cdots,q_9\}\in \mathbb{C}, \ \forall \ i\in \{1,\cdots,5\}, q_i=-q_{10-i} \ \& \ rank(D)=8\}.\]

For $(q_1,\cdots,q_9 )\in M^5$, since $rank(D)=8$, the solutions for the equation system (3.3) form a one-dimensional solution space. Similar to the proof above, we can get the function $P_\tau:M^5 \rightarrow \text{CP}^8$ defined by $P_\tau (q)=[R_{\tau,1}:\cdots:R_{\tau,9}] \in \text{CP}^8$ with $R_{\tau,1},\cdots,R_{\tau,9}$ being a nonzero solution of (3.3) for $q=(q_1,\cdots,q_9)\in M^5 \subset (M^2 \cap M^3)\cup(M^4 \setminus M^3)$.

Now let 
\begin{eqnarray}
B=\{[R_{\tau,1}:\cdots:R_{\tau,9}]\in \text{CP}^8 | \forall i\in \{1,8\},R_{\tau,i} \neq 0 \quad \& \nonumber \\
4R_{\tau,6} R_{\tau,2}^3 - R_{\tau,2}^2 R_{\tau,4}^2 + 4 R_{\tau,4}^3 R_{\tau,1} + 27 R_{\tau,6}^2 R_{\tau,1}^2 - 18 R_{\tau,2} R_{\tau,4} R_{\tau,1} R_{\tau,6} \neq 0\} \nonumber.
\end{eqnarray}

Let $M^6=\{q=(q_1,\cdots,q_9)\in M^5|P_\tau(q) \in B\}$.

Let $M^*=M^5\cap M^6$ and $E_0=M^* \times \mathbb{C}^9$ be the trivial complex vector bundle over $M^*$. Assume that $P_\tau \in H^0(M^*, \mathcal{O}(P(E_0)))$. So $P_\tau(q)=[R_{\tau,1}:\cdots:R_{\tau,9}] \in CP^8$ for $q \in M^*$.

Now we get a cubic curve embedded in $\text{CP}^2$ defined by
\begin{eqnarray}
C_{\tau,q}:F_{\tau,q}=R_{\tau,1} z^3+R_{\tau,2} xz^2+R_{\tau,4} x^2 z+R_{\tau,6} x^3+R_{\tau,8} xy^2=0.
\end{eqnarray}

For any $q\in M^*$, $P_\tau(q)=[R_{\tau,1}:\cdots:R_{\tau,9}]\in B$. The discriminant for the general homogeneous Weierstrass formula of the elliptic curve corresponding to equation (3.4) is 
\[\Delta=4R_{\tau,6} R_{\tau,2}^3 - R_{\tau,2}^2 R_{\tau,4}^2 + 4 R_{\tau,4}^3 R_{\tau,1} + 27 R_{\tau,6}^2 R_{\tau,1}^2 - 18 R_{\tau,2} R_{\tau,4} R_{\tau,1} R_{\tau,6} \neq 0.\] 
So equation (3.4) defines a smooth elliptic curve which is a submanifold of $\text{CP}^2$. In addition, from equation (3.2) and the equation system (3.3), the nine points in the set $\{f_{\tau} (f_{\tau,1} (q_1 )),\cdots,f_{\tau } (f_{\tau,9} (q_9 ))\}$ are on the smooth elliptic curve defined by the equation (3.4).

That is to say, for any $q\in M^*$ with $P_\tau(q)=[R_{\tau,1}:\cdots:R_{\tau,9}]\in \text{CP}^8$ and we can get a smooth elliptic curve $C_{\tau,q}\subset \text{CP}^2$ defined by equation (3.4). In addition, $C_{\tau,q}$ intersects $f_{\tau }(C_0(\tau))$ at the nine points in the set $\{f_{\tau} (f_{\tau,1} (q_1 )),\cdots,f_{\tau} (f_{\tau,9} (q_9 ))\}$.

Let $S_2=\{(\tau,q_1,\cdots,q_8) \in T| \forall \ i \in\{1,\cdots,5\}, q_i=-q_{10-i} \}$. Then $t \cdot F_{\tau,q} +s\cdot f_{\tau,\infty}=0$ gives a pencil of cubics on the fiber $\widetilde{\pi}^{-1}(\tau,q_1,\cdots,q_8)$ of the complex analytic family $(\widetilde{\mathcal{F}},T,\widetilde{\pi})$ for $(\tau,q_1,\cdots,q_8) \in S_2$. Now we can get an elliptic fibration $f_{\widetilde{\pi}^{-1} (S_2)}$ on $\widetilde{\pi}^{-1} (S_2)$. In addition, the restriction of the projectivized normal bundle of $M_i$ on $\widetilde{\pi}^{-1} (S_2)$ is a global holomorphic section of $f_{\widetilde{\pi}^{-1} (S_2)}$ for $i \in \{1,...,9\}$. 

Now if there are only singular fibers of Kodaira type $I_1$ and $C_{\tau,q}$ is a smooth elliptic curve for $(\tau,q_1,\cdots,q_8) \in S_2$, taking use of Corollary~\ref{cor3.2}, we have a holomorphic involution on $\widetilde{\pi}^{-1}(S_2)$ induced by the hypersurface involution on the rational elliptic surface with global sections.

However, for the other fibers, the equation system to give the analogous cubic curve of $F_{\tau,q}=0$ is different from (3.3). So we can not derive the analogous elliptic fibration and the holomorphic involution on the whole manifold $\widetilde{\mathcal{F}}$ directly. 

Finally, we claim the non-Steinness for all the compactifiable deformation families constructed in Section 2 and Section 3.

Now for $\tau \in U$, let $\pi_S: S \rightarrow \text{CP}^1$ be the unique elliptic fibration of $S \cong \text{CP}^2 \text{$\#$9} \overline{\text{CP}}^2$ induced by the pencil of cubics with the base points being the nine blow-up points. In addition, almost every fiber of the elliptic fibration $\pi_S$ is a smooth elliptic curve. So there exists $[s:t] \neq [1:0] \in \text{CP}^1$ such that $\pi_S^{-1} ([s:t]) \subset S\setminus C$ is a smooth elliptic curve. Then $S \setminus C$ is obviously not Stein. That is to say, all the fibers of the compactifiable deformation family $(\widetilde{\mathcal{F}},\widetilde{\mathcal{S}},T,\widetilde{\pi},o,j)$ are not Stein. Then $\widetilde{\mathcal{F}}\setminus \widetilde{\mathcal{S}}$ is not Stein.

Now the proofs of Corollary~\ref{cor1.2} and Corollary~\ref{cor1.3} have been completed through Corollary~\ref{cor3.1} and Corollary~\ref{cor3.2} respectively.

\section{Further discussions on Steinness}

\noindent
In this section, we confirm the non-Steinness of $S\setminus C$ in almost every case corresponding to the conjecture proposed by Marco Brunella in \cite{Marco:2010}, discuss the remaining problem and construct deformation families in the sense of \cite{Gasparim:2021} whose fibers were conjectured to be Stein by Andreas Höring and Thomas Peternel in \cite{Höring:2024}. Here, the nine blow-up points are allowed to be infinitely near. Moreover, taking use of the algebraic isomorphism derived from Lemma~\ref{lem2.1}, $f_\tau$ can be an arbitrary holomorphic embedding map.

\subsection{ The non-Steinness of the quasi-projective varieties $S\setminus C$ }

\noindent
Now let $\alpha$ and $\beta$ be smooth loops on the elliptic curve C corresponding to the line segments $[0,1]$ and $[0,\tau]$ on $\mathbb{C}/ \langle  1, \tau \rangle$. Let $e^{p\cdot 2\pi \sqrt{-1}}$ and $e^{q \cdot 2\pi \sqrt{-1}}$ be the monodromies of $N_{C/S}$ along the loops $\alpha$ and $\beta$ respectively with $(p,q) \in \mathbb{R}^2$. 

Firstly, if the nine blow-up points (may be infinitely near) satisfy the condition that $\sum_{i=1}^9 f_{\tau}^{-1} (p_i)=0$  mod  $\langle 1, \tau \rangle$, it is not hard to check that the statement for the non-Steinness of the quasi-projective $S \setminus C$ at the end of Section 3 is still right. So we do not have to give any more proofs here.

Secondly, we prove that if the nine blow-up points satisfy the condition that $\sum_{i=1}^9 f_{\tau}^{-1} (p_i)=p+q\tau$ mod $\langle 1, \tau \rangle$ with $(p,q)$ being a Diophantine number pair (see \cite[Definition~2.2]{xu:2025}), then the quasi-projective variety $S \setminus C$ is not Stein.

Let H be the hyperplane line bundle on $\text{CP}^2$. Taking use of the adjunction formula I in \cite[Chapter~1-1]{Phillip and Joseph:1994} and the proper transform for blowing up \cite[Chapter~4-1]{Phillip and Joseph:1994}, $\mathcal{O}_{f_\tau(C_0(\tau))} (3H-\sum_{i=1}^9 p_i ) \cong N_{C/S}$. Then the monodromies of the normal bundle $N_{C/S}$ along the loops $\alpha$ and $\beta$ are $e^{p\cdot 2\pi \sqrt{-1}}$ and $e^{q \cdot 2\pi \sqrt{-1}}$ respectively \cite[2.1]{Koike and Uehara:2022}. So the normal bundle satisfies the Diophantine condition.

Here, $N_{C/S}$ is topologically trivial and flat \cite{Ueda:1982}. Then we have the following lemma.

\begin{lemma}
\label{lem4.1}
If the nine blow-up points satisfy the condition that $\sum_{i=1}^9 f_{\tau}^{-1} (p_i)=p+q\tau$ mod $\langle 1, \tau \rangle$ with $(p,q)$ being a Diophantine number pair, then the quasi-projective variety $S \setminus C$ is not Stein.
\end{lemma}

\begin{proof}
From above, the normal bundle $N_{C/S}$ satisfies the Diophantine condition. Taking use of \cite[Theorem~1.7]{Koike and Uehara:2010}, there exists a pseudoflat neighborhood of the elliptic curve C of S. The boundary of the pseudoflat neighborhood is a 3-torus admitting a Levi-flat foliation with the dense leaf being either $\mathbb{C}$ or $\mathbb{C}^*$ \cite[Theorem~1.1]{Koike and Uehara:2010}. So $S \setminus C$ is not Stein. Then Lemma~\ref{lem4.1} is proved. 
\end{proof}                                                                         

Thirdly, it is to prove that if the nine blow-up points satisfy the condition that $\sum_{i=1}^9 f_{\tau}^{-1} (p_i)$ is of order $m$ with $m \geq 2$, then the quasi-projective variety $S \setminus C$ is not Stein.

Taking use of the logarithmic transformation proposed in  \cite{Kodaira:1963} by Kodaira, we can easily get the following lemma.

\begin{lemma}
\label{lem4.2}
If $\sum_{i=1}^9 f_{\tau}^{-1} (p_i)$ is of order $m \geq 2$, the quasi-projective variety $S \setminus C$ is not Stein.
\end{lemma} 

\begin{proof}
Taking use of \cite[Main Theorem~2.1-(2)]{Fujimoto:1990} and \cite[Proposition~1.1-(2)]{Fujimoto:1990}, if $\sum_{i=1}^9 f_{\tau}^{-1} (p_i)$ is of order $m$, then we can get an elliptic fibration $\Phi_{|-mK_s|}:S \rightarrow \text{CP}^1$ which is the anti-pluricanonical map and $S$ has one fiber of multiplicity $m$. In addition, $S$ can be derived from a rational elliptic surface $\widetilde{S}$ with global sections through the logarithmic transformation proposed in \cite{Kodaira:1963}.

Here, taking use of the method mentioned in \cite[Chapter 4-5]{Phillip and Joseph:1994}, we can construct the rational elliptic surface $\widetilde{S}$ above.

Let  $\omega_0$  be a local coordinate in a neighborhood $U_{[1:0]}$ of $[1:0]$ in $\text{CP}^1$. Let $\{ U_k\}_{k \in K}$ be an open cover of the submanifold $\Phi_{|-mK_s|}^{-1}(U_{[1:0]})$ of $S$ with $K$ being an index set and $U_k$ being suitably small open polydiscs. Then we can choose a holomorphic function $g_k$ on $U_k$ such that $g_k^m=\Phi_{|-mK_s|}^*(\omega_0)$ and $g_\alpha=e^{2\pi ik_{\alpha \beta}/m} \circ g_\beta$ for $k_{\alpha \beta} \in \{0,1,\cdots,m-1  \}$. 

Now the set $\Sigma$ collecting $(p,g_k)$ for $p \in U_k$ and $k \in K$ gives an unbranched m-sheeted cover of $\Phi_{|-mK_s|}^{-1}(U_{[1:0]})$ \cite[Chapter 4-5]{Phillip and Joseph:1994}. Let $\varphi$ be defined as $\varphi(p,g_k)= g_k(p)$ for $(p,g_k) \in \Sigma$ with $p \in U_k$ and $k\in K$. Then $\varphi$ gives an elliptic fibration on $\Sigma$.

Now we choose an m-section $e$ of the elliptic fibration $\Phi_{|-mK_s|}$. One typical example is the exceptional divisor. Moreover, let $\gamma$ denote the restriction of $e$ on $\Phi_{|-mK_s|}^{-1}(U_{[1:0]})$ . Here, $\gamma$ is an arc. Let $\widetilde{\gamma}$ denote a component of $\varphi^{-1}(\Phi_{|-mK_s|}(\gamma))$. Define an automorphism $f_{aut}$ on $\Sigma$ by
\[f_{aut} (p,g_\alpha )=(\widetilde{\gamma} \cdot g_\alpha^{-1}  (e^{2\pi i/m} \cdot g_\alpha (p)),e^{2 \pi i/m} \cdot g_\alpha)\]

\noindent
for $(p,g_\alpha )\in \Sigma$ with $p=\widetilde{\gamma}\cdot g_\alpha^{-1}  (g_\alpha (p))  \in U_{\alpha}$ and $\alpha \in K$. Here, $\widetilde{\gamma}\cdot g_\alpha^{-1}  (g_\alpha (p))$ denotes the intersection point of $\widetilde{\gamma}$ and $g_\alpha^{-1}  (g_\alpha (p))$. 

Now $\varphi^m$ gives an elliptic fibration on $\Sigma_0=\Sigma/\{\varphi^i\}$. Through gluing $\Sigma_0$ and $S\setminus \Phi_{|-mK_s|}^{-1} ([1:0])$, we can get a new surface $\widetilde{S}$. $\widetilde{S}$ is a rational elliptic surface with global sections. The inverse operation of the process above is the logarithmic transformation proposed in  \cite{Kodaira:1963}.

Here, $\widetilde{S}\setminus (\varphi^m)^{-1} ([1:0])$ is holomorphically isomorphic to $S \setminus \Phi_{|-mK_s|}^{-1}([1:0])$. Taking use of the proof above, $\widetilde{S}\setminus(\varphi^m)^{-1} ([1:0])$ is not Stein. Then $S \setminus \Phi_{|-mK_s |}^{-1} ([1:0])$ is not Stein. So Lemma~\ref{lem4.2} is proved.   
\end{proof}         

Here, $S$ can be generated by a Halphen pencil of index $m$ \cite[Section~3.1]{Zanardini:2008}. 

Finally, collecting the results above, we can get the following proposition.

\begin{proposition}
\label{pro4.3}
If $\sum_{i=1}^9 f_\tau^{-1} (p_i)=p+q \tau$ mod $\langle 1, \tau \rangle$ with $p+q\tau=0$, of order $m \geq 2$ or $(p,q)$ being a Diophantine number pair, the quasi-projective variety $S\setminus C$ is not Stein.
\end{proposition}

\subsection{Some discussions on Marco Brunella's conjecture}

\noindent
In this section, we give more discussions on the conjecture for the non-Steinness of the quasi-projective varieties $S\setminus C$ from the blow-up of $\text{CP}^2$ at nine points proposed in \cite{Marco:2010} by Marco Brunella.

Firstly, we can construct a smooth Hermitian metric on the anticanonical line bundle of $S$ whose curvature form is semipositive if $\sum_{i=1}^9 f_\tau^{-1} (p_i)=p+q\tau$ mod $\langle 1, \tau \rangle$ with $p+q\tau=0$, of order $m$ or $(p,q)$ being a Diophantine number pair.

\begin{proposition}
\label{pro4.4}
If $\sum_{i=1}^9 f_\tau^{-1} (p_i)=p+q\tau$ mod $\langle 1, \tau \rangle$ with $p+q\tau=0$, of order $m$ or $(p,q)$ being a Diophantine number pair, then there exists a smooth Hermitian metric on $K_S^{-1}$ with semipositive curvature.
\end{proposition}

\begin{proof}
Firstly, it is to prove that if $\sum_{i=1}^9 f_\tau^{-1} (p_i)=p+q\tau$ mod $\langle 1, \tau \rangle$ with $p+q\tau=0$, then there exists a smooth Hermitian metric on $K_S^{-1}$ with semipositive curvature.

Here, let $f_S:S \rightarrow \text{CP}^2$ be the blow-up map. Let $E_i$ be the line bundle on $S$ corresponding to the exceptional divisors $f_S^{-1} (p_i)$ for $i \in \{1,\cdots,9\}$ respectively. Taking use of Lemma in \cite[Chapter~1-4]{Phillip and Joseph:1994} and the formula for the proper transform in \cite[Chapter~4-1]{Phillip and Joseph:1994}, $-K_s \cong [C] \cong f_S^* (3H) \otimes(- E_1)\otimes \cdots \otimes(-E_9)$.

Since $f^* \mathcal{O}_{\text{CP}^1 } (1) \cong [C], -K_s \cong f^* \mathcal{O}_{\text{CP}^1 } (1)$. Here, the Fubini-Study metric h on the hyperplane line bundle $\mathcal{O}_{\text{CP}^1 } (1)$ is a Hermitian metric with positive curvature. So $f^* h$ is a smooth Hermitian metric on $-K_s$ with a semipositive (1,1)-form as curvature.

Secondly, it is to prove that if  $\sum_{i=1}^9 f_\tau^{-1} (p_i)=p+q\tau$ mod $\langle 1, \tau \rangle$ of order $m$, then there exists a smooth Hermitian metric on $K_S^{-1}$ with semipositive curvature. Marco Brunella mentioned the construction in \cite{Marco:2010}.

From the proof of Lemma~\ref{lem4.2}, we can construct a rational elliptic surface $\widetilde{S}$ with global sections through the same process as the proof of Lemma~\ref{lem4.2}. The elliptic fibration $\Phi_{|-mK_s|}:S \rightarrow \text{CP}^1$ is isomorphic to the elliptic fibration on $\widetilde{S}$ \cite[Proposition~1.1-(2)]{Fujimoto:1990}. Then $(-K_S)^{\otimes m} \cong [\Phi_{|-mK_s|}^{-1} ([1:0])] \cong \Phi_{|-mK_s |}^* \mathcal{O}_{\text{CP}^1 } (1)$. Taking use of the same method as above, $\Phi_{|-mK_s |}^* h$ is a smooth Hermitian metric on $(-K_S)^{\otimes m}$ with a semipositive (1,1)-form as curvature. Then $(\Phi_{|-mK_s|}^* h)^{1/m}$ is a smooth Hermitian metric on $-K_S$ with a semipositive (1,1)-form as curvature.

Finally, if the nine blow-up points satisfy the condition that $\sum_{i=1}^9 f_\tau^{-1} (p_i)=p+q\tau$  mod $\langle 1, \tau \rangle$ with $(p,q)$ being a Diophantine number pair, Marco Brunella has proved in \cite{Marco:2010} that there exists a smooth Hermitian metric on $K_S^{-1}$ with semipositive curvature.

In conclusion, if $\sum_{i=1}^9 f_\tau^{-1} (p_i)=p+q\tau$ mod $\langle 1, \tau \rangle$ with $p+q\tau=0$, of order $m$ or $(p,q)$ being a Diophantine number pair, there exists a smooth Hermitian metric on $K_S^{-1}$ with semipositive curvature. So Proposition~\ref{pro4.4} is proved. 

\end{proof}

Taking use of the result from Yau's paper \cite{Yau:1978} and Proposition~\ref{pro4.4}, if $\sum_{i=1}^9 f_\tau^{-1} (p_i)=p+q\tau$ mod $\langle 1, \tau \rangle$ with $p+q\tau=0$, of order $m$ or $(p,q)$ being a Diophantine number pair, then there exists a smooth Kähler metric on $S$ with semipositive Ricci curvature.

Moreover, taking use of \cite[Theorem~1]{Marco:2010}, under a certain condition for the normal bundle $N_{C/S}$ (For example, if the nine blow-up points satisfy a specific condition proposed in \cite[4.14]{Ogus:1976}) and assume that $S\setminus C$ does not contain any compact complex curve, the existence of Kähler metric on S with semipositive Ricci curvature is equivalent to the existence of the pseudoflat neighborhood of C in S which indicates the non-Steinness of $S\setminus C$.

Secondly, we construct a deformation family of quasi-projective varieties to illustrate the remaining work for the conjecture proposed by Marco Brunella \cite{Marco:2010}.

Choosing $i \in \{1,\cdots,9\}$, let $V_i$ be a sufficiently small disc neighborhood of $p_i$ and $(p_1,\cdots,p_9 ) \in \mathbb{C}^9$.

For $\tau \in Y$, let $f_{\tau,pro}$ be the natural projection map from $\mathbb{C}$ to $\mathbb{C}/\langle 1,\tau \rangle$. Let $\pi_{\text{CP}^2 \times V_i }:\text{CP}^2 \times V_i \rightarrow V_i$ also be the natural projection map. Now for any $\widetilde{p}_i \in V_i$, we blow up nine points in the set
\[\{f_\tau (f_{\tau,pro} (p_1)), \cdots,f_\tau (f_{\tau,pro} (\widetilde{p}_i)),\cdots, f_\tau (f_{\tau,pro} (p_9))\} \times \{\widetilde{p}_i\} \subset f_\tau (C_0 (\tau))\times \{\widetilde{p}_i\}\]
on the fiber $\pi_{\text{CP}^2 \times V_i}^{-1} (\widetilde{p}_i) \cong \text{CP}^2$.

Then we can get a complex manifold $\widetilde{\mathcal{F}}_1$ and a projection map $\pi_{\widetilde{\mathcal{F}}_1}:\widetilde{\mathcal{F}}_1 \rightarrow V_i$ induced by $\pi_{\text{CP}^2 \times V_i}$. Let $S_3=\left \{(x,y) \in f_\tau (C_0 (\tau)) \times V_i|x \in f_\tau (C_0 (\tau))\right \}$ and $\widetilde{\mathcal{S}}_3$ be the strict transform of $S_3$. 

Through the same proof as Theorem~\ref{thm2.4}, we can get the fact that $(\widetilde{\mathcal{F}}_1,V_i,\pi_{\widetilde{\mathcal{F}}_1})$ is a complex analytic family of $\text{CP}^2 \text{$\#$9} \overline{\text{CP}}^2$ and $(\widetilde{\mathcal{S}}_3,V_i,\pi_{\widetilde{\mathcal{F}}_1 } |_{\widetilde{\mathcal{S}}_3})$ is a complex analytic family of elliptic curves.

Let $S_{origin}$ be the blow-up of $\text{CP}^2$ at the nine points in the set 
\[\{f_{\tau } (f_{\tau,pro} (p_1 )), \cdots,f_{\tau } (f_{\tau,pro} (p_9 ))\} \subset f_{\tau } (C_0 (\tau ))\] 
and $C_{origin}$ be the strict transform of $f_{\tau } (C_0 (\tau ))$. Let $x_{origin}=p_i \in V_i$ and $j_{origin}:S_{origin}\setminus C_{origin}\rightarrow \pi_{\widetilde{\mathcal{F}}_1}^{-1} (x_{origin} )$ be defined as $j_{origin}(x)=(x,p_i)$ for $x \in S_{origin}\setminus C_{origin}$.

Through the same proof as Theorem~\ref{thm2.6} , $(\widetilde{\mathcal{F}}_1,\widetilde{\mathcal{S}}_3,V_i,\pi_{\widetilde{\mathcal{F}}_1},x_{origin},j_{origin})$ is a compactifiable deformation family of $S_{origin}\setminus C_{origin}$ for a suitable $V_i$ chosen.

Taking use of Proposition~\ref{pro4.3}, almost all the fibers of the compactifiable deformation family $(\widetilde{\mathcal{F}}_1,\widetilde{\mathcal{S}}_3,V_i,\pi_{\widetilde{\mathcal{F}}_1},x_{origin},j_{origin})$ are not Stein. Taking use of Proposition~\ref{pro4.4}, the anticanonical line bundle is semipositive for almost all the fibers of the complex analytic family $(\widetilde{\mathcal{F}}_1,V_i,\pi_{\widetilde{\mathcal{F}}_1})$.

Now besides the non-Steinness conjecture by Marco Brunella, it is also natural to ask whether the anticanonical line bundle is semipositive for all the fibers of the complex analytic family $(\widetilde{\mathcal{F}}_1,V_i,\pi_{\widetilde{\mathcal{F}}_1})$.

There is a conjecture by Takayuki Koike \cite[Conjecture~2.1]{Takayuki:2023} which has a close relationship with this question. However, the answer also depends on the Ueda type defined in \cite{Ueda:1982} for the strict transform of the elliptic curve embedded in each fiber.

\subsection{A conjecture for deformation families with Stein fibers}

\noindent
In this section, we give another kind of deformation family in the sense of \cite{Gasparim:2021} whose fibers were conjectured to be Stein in \cite{Höring:2024}.

Firstly, we should introduce the definition for the canonical extension $Z_M$ of a Kähler manifold M \cite{Greb and wong:2020}. Let $\omega$ be a Kähler form on M which defines a class $a_{\omega} \in H^1 (M,T^* M)$ with $\alpha_{\omega} \in \mathcal{A}^{0,1}(T^* M)$ being a Dolbeault representative. We have the following extension of vector bundles
\[  0 \rightarrow T^* M \rightarrow W \rightarrow \mathcal{O}_M \rightarrow 0   \]
corresponding to $a_{\omega}$. $W=T^* M \oplus \mathcal{O}_M$ has a holomorphic structure as 
\[ \overline{\partial}_W:=\begin{bmatrix} \overline{\partial}_{T^* M} & \alpha_{\omega} \\  & \overline{\partial}_{\mathcal{O}_M} \end{bmatrix}.\]

Now define $W_M=P(W)\setminus P(T^* M)$. Let $(\mathcal{M},\pi_\mathcal{M},B)$ be a nontrivial complex analytic family with B being a sufficiently small polydisc centered at 0 and $\pi_\mathcal{M}^{-1} (0)$ is holomorphically isomorphic to M.

Then we have the following proposition.

\begin{proposition}
\label{pro4.5}
If there exists a polydisc $B_1 \Subset B$ such that $\mathcal{M}_{B_1}=\pi_\mathcal{M}^{-1}(B_1)$ is Kähler, then we can construct a deformation family of the canonical extension $W_M$ of the compact complex manifold M in the sense of \cite{Gasparim:2021} with the base manifold being $B_1$.
\end{proposition}

\begin{proof}
In this proof, the deformation families are all in the sense of \cite{Gasparim:2021}. Let $m=\text{dim}(\mathcal{M}_{B_1})$ and $n=\text{dim}(B_1)$. 

Let $\pi_{B_1}=\pi_\mathcal{M} |_{\pi_\mathcal{M}^{-1}(B_1)}$. Taking use of the inverse function theorem, we can get an atlas $\{(X_i,z_i)\}_{i \in V}$ on $\mathcal{M}_{B_1}$ with $z_i(p)=(z_i^1 (p), \cdots ,z_i^{m-n} (p),t_{m-n+1}, \cdots ,t_m)$ and $(t_{m-n+1},\cdots,t_m )$ $=\pi_{B_1} (p)$ for $p \in X_i$ and $i \in V$. 

Now since $\mathcal{M}_{B_1}=\pi_\mathcal{M}^{-1}(B_1)$ is Kähler, we can construct a Kähler form $\omega_{B_1}$ on $\mathcal{M}_{B_1}$ such that the restriction of $\omega_{B_1}$ on each fiber is also a Kähler form and $\omega_{B_1}|_{\pi_{B_1}^{-1} (0)}$ is the pullback of the Kähler form $\omega$ on M. Here, $\omega_{B_1}$ defines a class $a_{B_1} \in H^1 (\mathcal{M}_{B_1},T^* (\mathcal{M}_{B_1}/B_1))$ with $\alpha_{B_1} \in \mathcal{A}^{0,1}(\mathcal{M}_{B_1},T^* (\mathcal{M}_{B_1}/B_1))$ being a Dolbeault representative. Then we have the following extension of vector bundles
\[0 \rightarrow T^* (\mathcal{M}_{B_1}/B_1) \rightarrow W_{B_1} \rightarrow \mathcal{O}_{\mathcal{M}_{B_1}} \rightarrow 0\]
corresponding to $a_{B_1}$. Now define
\[W_{\mathcal{M}_{B_1}}=P(W_{B_1})\setminus P(T^* (\mathcal{M}_{B_1}/B_1)).\]
In addition, $W_{B_1}=T^* (\mathcal{M}_{B_1}/B_1) \oplus \mathcal{O}_{\mathcal{M}_{B_1}}$ \cite[1.3]{Greb and wong:2020} with a holomorphic structure as 
\[ \overline{\partial}_{W_{B_1}}:=\begin{bmatrix} \overline{\partial}_{T^* (\mathcal{M}_{B_1}/B_1)} & \alpha_{B_1} \\  & \overline{\partial}_{\mathcal{O}_{\mathcal{M}_{B_1}}} \end{bmatrix}.\]

Let $\pi_{co}:T^* (\mathcal{M}_{B_1}/B_1) \rightarrow \mathcal{M}_{B_1}$ be the natural projection. So $\pi_{B_1}\circ \pi_{co}$ is a holomorphically surjective submersion.  

Here, $(\pi_{B_1} \circ \pi_{co} )^{-1} (0)=T^* \pi_{B_1}^{-1} (0)$. It is easy to prove that $T^* (\mathcal{M}_{B_1}/B_1)$ is locally trivial in the $C^\infty$ category through the local trivializations of $T^* (\mathcal{M}_{B_1}/B_1)$ and the $C^\infty$ local triviality of $\pi_{B_1}$ which can be derived from \cite[Theorem~2.5]{Kodaira:2005}. So $\pi_{B_1}\circ \pi_{co}:T^* (\mathcal{M}_{B_1}/B_1) \rightarrow B_1$ is a deformation family of $T^* \pi_{B_1}^{-1} (0)$ with $(\pi_{B_1} \circ \pi_{co})^{-1} (t)=T^* \pi^{-1}_{B_1 } (t)$.

Now $\mathcal{O}_{\mathcal{M}_{B_1}} |_{\pi_{B_1}^{-1} (t)}=\mathcal{O}_{\pi_{B_1 }^{-1} (t)}$ for $t \in B_1$. From above, $\omega_{B_1} |_{\pi_{B_1}^{-1} (t)}$ is a Kähler form which defines a class $a_t=a_{B_1}|_{\pi_{B_1}^{-1}(t)} \in H^1 (\pi_{{B_1}}^{-1} (t),T^* \pi_{B_1}^{-1} (t))$ with $\alpha_t=\alpha_{B_1}|_{\pi_{B_1}^{-1}(t)} \in \mathcal{A}^{0,1}(T^* (\pi_{B_1}^{-1} (t)))$ being a Dolbeault representative.

Let $\pi_{W_1} :W_{B_1}\rightarrow \mathcal{M}_{B_1}$ be the natural projection induced by $\pi_{co}$. Then $\pi_{B_1} \circ \pi_{W_1}:W_{B_1} \rightarrow B_1$ is a deformation family of $T^* \pi_{B_1}^{-1} (0) \oplus \mathcal{O}_{\pi_{B_1}^{-1} (0)}$ with $(\pi_{B_1} \circ \pi_{W_1} )^{-1} (t)=T^* \pi^{-1}_{B_1}(t) \oplus \mathcal{O}_{\pi_{B_1}^{-1}(t) }$.
Here, the holomorphic structure on $T^* \pi^{-1}_{B_1}(t) \oplus \mathcal{O}_{\pi_{B_1}^{-1}(t) }$ is similar to the holomorphic structure of $W$ above induced by $\alpha_t$, $\overline{\partial}_{T^* \pi_{B_1}^{-1} (t)}$ and $\overline{\partial}_{\mathcal{O}_{\pi_{B_1}^{-1} (t)}}$. Then we have the following extension of vector bundles
\[0 \rightarrow T^* \pi_{B_1}^{-1} (t) \rightarrow T^* \pi_{B_1}^{-1} (t) \oplus \mathcal{O}_{\pi_{B_1}^{-1} (t)} \rightarrow \mathcal{O}_{\pi_{B_1}^{-1} (t)} \rightarrow 0 \]
corresponding to $ a_t$.

Now let $\pi_P:P(W_{B_1})\setminus P(T^* (\mathcal{M}_{B_1}/B_1)) \rightarrow \mathcal{M}_{B_1}$ be the natural projection induced by $\pi_{W_1}$ and $\pi_{co}$.

Then $\pi_{B_1} \circ \pi_P:P(W_{B_1})\setminus P(T^* (\mathcal{M}_{B_1}/B_1)) \rightarrow B_1$ is a deformation family of $P(T^* \pi_{B_1}^{-1} (0) \oplus \mathcal{O}_{\pi_{B_1}^{-1} (0)})\setminus P(T^* \pi_{B_1}^{-1}(0))$ with 
\[(\pi_{B_1} \circ \pi_P )^{-1} (t)=P(T^* \pi_{B_1}^{-1} (t) \oplus \mathcal{O}_{\pi_{B_1}^{-1} (t) })\setminus P(T^* \pi_{B_1}^{-1} (t)). \]

$(\pi_{B_1} \circ \pi_P)^{-1} (0)=P(T^* \pi_{B_1}^{-1} (0) \oplus \mathcal{O}_{\pi_{B_1}^{-1} (0) })\setminus P(T^* \pi_{B_1}^{-1} (0))$ is holomorphically isomorphic to $W_M=P(W)\setminus P(T^*M)$.

Here, $P(T^* \pi_{B_1}^{-1}(t) \oplus \mathcal{O}_{\pi_{B_1}^{-1} (t) })\setminus P(T^* \pi_{B_1}^{-1} (t))$ is the canonical extension of $\pi_{B_1}^{-1} (t)$ for $t \in B_1$.

Therefore, $\pi_{B_1} \circ \pi_P :P(W_{B_1})\setminus P(T^* (\mathcal{M}_{B_1}/B_1)) \rightarrow B_1$ is a deformation family of the canonical extension $W_M$ of the compact complex manifold M in the sense of Edoardo Ballico, Elizabeth Gasparim and Francisco Rubilar \cite{Gasparim:2021} with the base manifold being $B_1$. So Proposition~\ref{pro4.5} is proved. 
\end{proof}

In fact, the condition above can be replaced by the other equivalent condition proposed by Jian Chen in \cite[Theorem~1.2]{Chen:2026}.

Here, Andreas Höring and Thomas Peternell attempted to conjecture that the fibers for the deformation family of the canonical extension $W_M$ of M constructed above are all Stein if and only if the tangent bundle for each fiber of the complex analytic family $(\mathcal{M}_{B_1},\pi_{B_1},B_1)$ is numerically effective in \cite[Conjecture~1.1]{Höring:2024}. A recent result corresponding to the conjecture is \cite[Theorem~0.2.]{Müller:2025}.

Here, some compact complex manifolds with numerically effective tangent bundles such as Moishezon manifolds, surfaces and Kähler 3-folds with nef tangent bundles were studied in \cite{Schneider:1994} by Jean-Pierre Demailly, Thomas Peternell and Michael Schneider.

One typical example of the deformation family is the deformation family of the canonical extension of a torus. It is obvious that $S_1 \subset \text{CP}^2 \times U$ is Kähler. Similar to the proof of Proposition 4.5 above, we can construct a deformation family of the canonical extension of a torus in the sense of \cite{Gasparim:2021}.

Let $\omega_{s_1}$ be a Kähler form on $S_1$ which defines a class $a_{s_1} \in H^1 (S_1,T^* (S_1/U))$ with $\alpha_{s_1} \in \mathcal{A}^{0,1}(S_1,T^* (S_1/U))$ being a Dolbeault representative. Let $\pi_{s_1}=\pi|_{S_1}$. Then $\omega_{s_1} |_{\pi_{s_1}^{-1} (t)}$ is a Kähler form on $\pi_{s_1}^{-1} (t)$ which defines a class $a_{s_1}|_{\pi_{s_1}^{-1} (t)} \in H^1 (\pi_{s_1}^{-1} (t),T^* \pi_{s_1}^{-1} (t))$ with $\alpha_{s_1}|_{\pi_{s_1}^{-1} (t)} \in \mathcal{A}^{0,1}(T^* (\pi_{s_1}^{-1} (t)))$ being a Dolbeault representative.
 
Now we have an extension of vector bundles
\[0 \rightarrow T^* (S_1/U) \rightarrow W_U \rightarrow \mathcal{O}_{S_1} \rightarrow 0\]
corresponding to $a_{s_1}$. Here, $W_U=T^* (S_1/U) \oplus \mathcal{O}_{S_1}$ with a holomorphic structure as 
\[ \overline{\partial}_{W_U}:=\begin{bmatrix} \overline{\partial}_{T^* (S_1/U)} & \alpha_{s_1} \\  & \overline{\partial}_{\mathcal{O}_{S_1}} \end{bmatrix}.\]
Let
\[W_{s_1}=P(W_U)\setminus P(T^*(S_1/U) ).\]
\noindent
Let $\pi_U:T^* (S_1/U) \rightarrow S_1$ be the natural projection. Here, 
\[(\pi_{s_1} \circ \pi_U )^{-1} (\tau_0 )=T^* \pi_{s_1}^{-1} (\tau_0).\]
\noindent
Let $\pi_W:W_U\rightarrow S_1$ be the natural projection induced by $\pi_U$. Let $\pi_{W_{s_1}}:P(W_U)\setminus P(T^* (S_1/U)) \rightarrow S_1$ be the natural projection induced by $\pi_W$ and $\pi_U$. Similar to the proof above, we can confirm that 
\[\pi_{s_1}\circ \pi_{W_{s_1}}:P(W_U )\setminus P(T^* (S_1/U) ) \rightarrow U\]
is a deformation family of the canonical extension of a torus. Taking use of \cite[Proposition~2.13]{Greb and wong:2020}, 
all fibers for the deformation family of the canonical extension of a torus constructed above are Stein.

\affiliationone{
Fan Xu \\
\vspace{0.15cm}
School of Mathematics, \\
\vspace{0.15cm}
Sun Yat-sen University.\\
\vspace{0.15cm}
Guangzhou, PR China,\\
\vspace{0.15cm}
510275\\
\vspace{0.15cm}
\email{xufan23@mail.sysu.edu.cn}}



\end{document}